\documentclass[reqno,11pt]{amsart}

\usepackage{a4wide}
\usepackage{color}
\usepackage{mathrsfs}
\usepackage{mathtools}
\usepackage{amsmath}
\usepackage{amssymb}
\usepackage{bbm}
\usepackage{tikz-cd}
\usepackage{esint}
\usepackage{nicefrac}
\numberwithin{equation}{section}
\usepackage[colorlinks,citecolor=green,linkcolor=red]{hyperref}

\usepackage[latin1]{inputenc}

\newcommand{\N}{\mathbb{N}}
\newcommand{\R}{\mathbb{R}}
\newcommand{\B}{\mathbb{B}}
\newcommand{\sfd}{{\sf d}}
\renewcommand{\d}{{\mathrm d}}
\newcommand{\X}{{\rm X}}
\newcommand{\Y}{{\rm Y}}
\newcommand{\mm}{\mathfrak{m}}
\newcommand{\1}{\mathbbm 1}
\newcommand{\limi}{\varliminf}
\newcommand{\lims}{\varlimsup}

\newcommand{\fr}{\penalty-20\null\hfill\(\blacksquare\)}

\newtheorem{theorem}{Theorem}[section]

\newtheorem{lemma}[theorem]{Lemma}
\newtheorem{proposition}[theorem]{Proposition}
\newtheorem{definition}[theorem]{Definition}

\newtheorem{remark}[theorem]{Remark}

\title{On the reflexivity properties of Banach bundles and Banach modules}

\author[Milica Lu\v{c}i\'{c}]{Milica Lu\v{c}i\'{c}}
\address[Milica Lu\v{c}i\'{c}]{Department of Mathematics and Informatics, Faculty of Sciences, University of Novi Sad,
Trg D.\ Obradovi\'{c}a 4, 21000 Novi Sad, Serbia}
\email{milica.lucic@dmi.uns.ac.rs}

\author[Enrico Pasqualetto]{Enrico Pasqualetto}
\address[Enrico Pasqualetto]{Scuola Normale Superiore, Piazza dei Cavalieri 7, 56126 Pisa, Italy}
\email{enrico.pasqualetto@sns.it}

\author[Ivana Vojnovi\'{c}]{Ivana Vojnovi\'{c}}
\address[Ivana Vojnovi\'{c}]{Department of Mathematics and Informatics, Faculty of Sciences, University of Novi Sad,
Trg D.\ Obradovi\'{c}a 4, 21000 Novi Sad, Serbia}
\email{ivana.vojnovic@dmi.uns.ac.rs}
\begin{document}
\date{\today}
\keywords{Banach bundle, normed module, uniform convexity, reflexivity}
\subjclass[2020]{18F15, 53C23}
\begin{abstract}
In this paper we investigate some reflexivity-type properties of separable measurable Banach bundles
over a \(\sigma\)-finite measure space. Our two main results are the following:
\begin{itemize}
\item The fibers of a bundle are uniformly convex (with a common modulus of convexity) if and only
if the space of its \(L^p\)-sections is uniformly convex for every \(p\in(1,\infty)\).
\item The fibers of a bundle are reflexive if and only if the space of its \(L^p\)-sections is reflexive.
\end{itemize}
These results generalise the well-known corresponding ones for Lebesgue--Bochner spaces.
\end{abstract}
\maketitle
\tableofcontents
\section{Introduction}
\subsection*{General overview}
In this paper, we focus on the theory of measurable Banach bundles over a given \(\sigma\)-finite measure space \((\X,\Sigma,\mm)\).
Our primary aim is to understand whether some important properties of the fibers of a measurable Banach bundle (such as Hilbertianity,
uniform convexity, and reflexivity) carry over to the space of its \(L^p\)-sections, and vice versa.
\medskip

Given an `ambient' Banach space \(\B\), a weakly measurable multivalued map \(\mathbf{E}\colon\X\twoheadrightarrow\B\) is said
to be a \emph{Banach \(\B\)-bundle} on \(\X\) if \(\mathbf{E}(x)\) is a closed linear subspace of \(\B\) for every \(x\in\X\).
A strongly measurable map \(v\colon\X\to\B\) such that \(v(x)\in\mathbf{E}(x)\) for every \(x\in\X\) is called a \emph{section} of \(\mathbf{E}\).
For any exponent \(p\in(1,\infty)\), we denote by \(\Gamma_p(\mathbf{E})\) the space of (equivalence classes, up to \(\mm\)-a.e.\ equality, of)
those sections of \(\mathbf{E}\) for which \(\X\ni x\mapsto\|v(x)\|_\B\in\R\) belongs to \(L^p(\mm)\). It is worth pointing out that the well-known
concept of \emph{Lebesgue--Bochner space} \(L^p(\mm;\B)\) is a particular instance of a section space, corresponding to
the bundle constantly equal to \(\B\).
\medskip

The space \(\Gamma_p(\mathbf{E})\) naturally comes with a pointwise multiplication by \(L^\infty(\mm)\)-functions and with a
\emph{pointwise norm} operator \(|\cdot|\colon\Gamma_p(\mathbf{E})\to L^p(\mm)\), given by \(|v|\coloneqq\big\|v(\cdot)\|_\B\).
Furthermore, the function \(\Gamma_p(\mathbf{E})\ni v\mapsto\|v\|_{\Gamma_p(\mathbf{E})}\coloneqq\big\||v|\big\|_{L^p(\mm)}\in[0,+\infty)\)
defines a complete norm on \(\Gamma_p(\mathbf{E})\). All in all, \(\Gamma_p(\mathbf{E})\) is an \emph{\(L^p(\mm)\)-normed \(L^\infty(\mm)\)-module},
in the sense of Gigli \cite{Gigli14}. We remark that, more surprisingly, in the case of separable normed modules the converse implication holds
as well: every separable \(L^p(\mm)\)-normed \(L^\infty(\mm)\)-module \(\mathscr M\) is isomorphic to \(\Gamma_p(\mathbf{E})\), for some measurable
Banach \(\B\)-bundle \(\mathbf{E}\) on \(\X\), where \(\B\) is a separable Banach space. This representation result -- first obtained
in \cite{LP18} for `locally finitely-generated' modules and later generalised in \cite{DMLP21} to all separable modules -- in fact strongly
motivates our interest towards the language of measurable Banach bundles.
\medskip

The theory of \(L^p(\mm)\)-normed \(L^\infty(\mm)\)-modules was introduced by Gigli in \cite{Gigli14} -- as already mentioned -- and
refined further in \cite{Gigli17}. The main purpose was to provide a robust functional-analytic framework, suitable for constructing
effective notions of \emph{\(1\)-forms} and \emph{vector fields} in the setting of metric measure spaces. The key object introduced in \cite{Gigli14}
is the \emph{cotangent module} \(L^2(T^*\X)\), which is obtained, roughly speaking, as the completion of the \(L^\infty(\mm)\)-linear combinations
of the `formal differentials' \(\d f\) of Sobolev functions \(f\in W^{1,2}(\X)\). It is evident that it is not sufficient to consider only the Banach
space structure of \(L^2(T^*\X)\), but instead one has to keep track also of the `pointwise' behaviour of the elements of \(L^2(T^*\X)\),
which is encoded into the \(L^\infty(\mm)\)-module structure and the pointwise norm. Due to this reason, \(L^p(\mm)\)-normed \(L^\infty(\mm)\)-modules
were the correct class of spaces to take into account. In this regard, an enlightening side result -- which is not strictly needed for the
purposes of this paper, but that we report for future reference -- will be discussed in Appendix \ref{app:criterion_nmod}. More precisely,
in Theorem \ref{thm:when_norm_ind_ptwse_norm} we will characterise those complete norms over a given \(L^\infty(\mm)\)-module that are induced by
an \(L^p(\mm)\)-valued pointwise norm operator via integration.
\medskip

Prior to the development of \(L^p\)-normed \(L^\infty\)-modules on metric measure spaces, some strictly related notions were already
well-established in the literature, for instance \emph{randomly normed spaces} \cite{HLR91} or \emph{random normed modules} \cite{Guo-1995},
which are typically formulated over a probability measure space. In view of this fact, we will work in the general framework of
normed modules over a \(\sigma\)-finite measure space. The notion of a random normed module is an important concept in random metric theory,
which is derived from the investigation of probabilistic metric spaces. A key construction in this theory is that of a random conjugate space.
The random metric theory has applications in finance optimisation problems, and it is connected with the study of conditional and dynamic risk
measures. See \cite{Guo-2011} and the references therein.
\subsection*{Statement of results}
Let us now describe more in details the main results that we will achieve in this paper. Fix a \(\sigma\)-finite measure space \((\X,\Sigma,\mm)\),
a separable Banach space \(\B\), and a measurable Banach \(\B\)-bundle \(\mathbf{E}\) on \(\X\). Then we will prove the following statements:
\begin{itemize}
\item[\(\rm a)\)] \emph{\(\mathbf{E}(x)\) is Hilbert for \(\mm\)-a.e.\ \(x\in\X\) if and only if \(\Gamma_2(\mathbf{E})\) is Hilbert.}
See Theorem \ref{thm:Hilb_bundles_mod}.
\item[\(\rm b)\)] \emph{\(\mathbf{E}(x)\) is uniformly convex for \(\mm\)-a.e.\ \(x\in\X\) (and with modulus of convexity independent of \(x\))
if and only if \(\Gamma_p(\mathbf{E})\) is uniformly convex for all \(p\in(1,\infty)\)}. See Theorem \ref{thm:unif_conv_bundles_mod}. Its proof
is more involved than the one for the Hilbertian case, and relies upon some previous results about \emph{random uniform convexity} by Guo and Zeng
\cite{GZRNM,GZRNM2}. The corresponding statement for Lebesgue--Bochner spaces can be found, \emph{e.g.}, in \cite{UCpaper}.
\item[\(\rm c)\)] \emph{\(\mathbf{E}(x)\) is reflexive for \(\mm\)-a.e.\ \(x\in\X\) if and only if \(\Gamma_p(\mathbf{E})\) is reflexive for all
\(p\in(1,\infty)\).} See Theorem \ref{thm:reflex_impl}.
\end{itemize}
The above results are well-known in the special case of Lebesgue--Bochner spaces. We point out that the implication `\(L^p(\mm;\B)\) reflexive
implies \(\B\) reflexive' can be easily proved: assuming \(\mm(\X)=1\) for simplicity, one can realise \(\B\) as a closed linear subspace of
\(L^p(\mm;\B)\) (by sending each \(v\in\B\) to the section constantly equal to \(v\)). However, the corresponding implication
`\(\Gamma_p(\mathbf{E})\) reflexive implies \(\mathbf{E}(x)\) reflexive for \(\mm\)-a.e.\ \(x\)' will require a much more difficult proof.
\medskip

We also mention that, along the way to prove item c), we will obtain a result of independent interest: shortly said, given a measurable Banach
\(\B\)-bundle \(\mathbf{E}\) (with \(\B\) not necessarily separable), the dual of \(\Gamma_p(\mathbf{E})\) as a normed module can be
identified with the space of \(q\)-integrable \emph{weakly\(^*\) measurable} sections of the dual bundle \(\X\ni x\mapsto\mathbf{E}(x)'\),
where \(\frac{1}{p}+\frac{1}{q}=1\). See Section \ref{s:dual_sections} for the precise formulation, as well as Theorem \ref{thm:char_dual}
for the relevant equivalence result. The corresponding statement for Lebesgue--Bochner spaces, stating that
\(L^p(\mm;\B)'\) can be identified with the space \(L^q_{w^*}(\mm;\B')\) of \(q\)-integrable `weakly\(^*\) measurable' maps
from \((\X,\Sigma,\mm)\) to \(\B'\), was previously known (see, \emph{e.g.}, \cite{dunford1958linear}). We also point out
that a variant of the statement in c) for normed modules has been recently obtained in \cite[Theorems 3.9 and 4.17]{GLP22}.
However, in general neither the results of \cite{GLP22} imply c), nor the vice versa.
\subsection*{Addendum}
While in a previous version of this manuscript only one of the two implications in c) was obtained (namely, that `reflexive fibers
implies reflexive section space'), in the current version the full equivalence is proved. This is due to the fact that an anonymous
colleague kindly pointed out to us the result \cite[Theorem 6.19]{HLR91}, which is the analogue of c) in the setting of \emph{direct integrals}.
However, we do not obtain the implication `reflexive section space implies reflexive fibers' as a consequence of \cite[Theorem 6.19]{HLR91},
but we rather follow the same proof strategy; see Remark \ref{rmk:on_proof_refl} for more comments on this. It would be very interesting
-- but outside the scopes of this manuscript -- to investigate the relation between our notion of Banach bundle
and the theory of direct integrals considered in \cite{HLR91}.
\subsection*{Acknowledgements}
The authors thank Nicola Gigli for having suggested the proof of Theorem \ref{thm:char_dual}.
The first and the third named authors acknowledge the financial support of the Ministry of Education,
Science and Technological Development of the Republic of Serbia (Grant No.\ 451-03-68/2022-14/200125).
The second named author acknowledges the support by the Balzan project led by Luigi Ambrosio. The third author acknowledges financial support of the Croatian Science Foundation under the project 2449 MiTPDE.
\section{Preliminaries}
To begin with, we fix some general terminology, which we will use throughout the entire paper. For any \(p\in[1,\infty]\),
we tacitly denote by \(q\in[1,\infty]\) its \textbf{conjugate exponent}, \emph{i.e.},
\[
\frac{1}{p}+\frac{1}{q}=1.
\]
Given a \(\sigma\)-finite measure space \((\X,\Sigma,\mm)\), we denote by \(\mathcal L^0_{\rm ext}(\Sigma)\) the space of all measurable
functions from \(\X\) to \(\R\cup\{\pm\infty\}\), while \(L^0_{\rm ext}(\mm)\) stands for the quotient of \(\mathcal L^0_{\rm ext}(\Sigma)\) up to
\(\mm\)-a.e.\ equality. We denote by \(\pi_\mm\colon\mathcal L^0_{\rm ext}(\Sigma)\to L^0_{\rm ext}(\mm)\) the usual projection map on the quotient.
Moreover, we define \(\mathcal L^0(\Sigma)\coloneqq\big\{f\in\mathcal L^0_{\rm ext}(\Sigma)\,:\,f(\X)\subseteq\R\big\}\) and
\(L^0(\mm)\coloneqq\pi_\mm\big(\mathcal L^0(\Sigma)\big)\). During the paper we will use two different notions of `essential supremum/infimum', namely:
\begin{itemize}
\item Given any \(f\in\mathcal L^0_{\rm ext}(\Sigma)\), we define \({\rm ess\,sup}_\X f,{\rm ess\,inf}_\X f\in\R\cup\{\pm\infty\}\) respectively as
\[\begin{split}
\underset{\X}{\rm ess\,sup}\,f&\coloneqq\inf\Big\{\lambda\in\R\cup\{\pm\infty\}\;\Big|\;f\leq\lambda,\text{ holds }\mm\text{-a.e.\ on }\X\Big\},\\
\underset{\X}{\rm ess\,inf}\,f&\coloneqq\sup\Big\{\lambda\in\R\cup\{\pm\infty\}\;\Big|\;f\geq\lambda,\text{ holds }\mm\text{-a.e.\ on }\X\Big\}.
\end{split}\]
\item Given a (possibly uncountable) family \(\{f_i\}_{i\in I}\subseteq\mathcal L^0_{\rm ext}(\Sigma)\), we define
\(\bigvee_{i\in I}f_i\in L^0_{\rm ext}(\mm)\) as the unique \(f\in L^0_{\rm ext}(\mm)\)
such that \(f_i\leq f\) \(\mm\)-a.e.\ for every \(i\in I\) and satisfying
\[
g\in L^0_{\rm ext}(\mm),\;\;f_i\leq g\;\,\mm\text{-a.e.\ for every }i\in I\quad\Longrightarrow\quad f\leq g\;\,\mm\text{-a.e.}.
\]
Similarly, \(\bigwedge_{i\in I}f_i\in L^0_{\rm ext}(\mm)\) is the unique element \(f\) of \(L^0_{\rm ext}(\mm)\)
such that \(f_i\geq f\) \(\mm\)-a.e.\ for every \(i\in I\) and satisfying
\[
g\in L^0_{\rm ext}(\mm),\;\;f_i\geq g\;\,\mm\text{-a.e.\ for every }i\in I\quad\Longrightarrow\quad f\geq g\;\,\mm\text{-a.e.}.
\]
\end{itemize}
Notice that the above notions of essential supremum/infimum are invariant under modifications of the functions \(f\) and \(f_i\) on an
\(\mm\)-negligible set, thus accordingly we can unambiguously consider \({\rm ess\,sup}_\X f\), \({\rm ess\,inf}_\X f\), \(\bigvee_{i\in I}f_i\),
\(\bigwedge_{i\in I}f_i\) whenever \(f\) and \(f_i\) are elements of \(L^0_{\rm ext}(\mm)\).
\medskip

The \textbf{Lebesgue spaces} are defined in the usual way: first, given any \(p\in[1,\infty)\), we define
\[
\mathcal L^p(\mm)\coloneqq\bigg\{f\in\mathcal L^0(\Sigma)\;\bigg|\;\int|f|^p\,\d\mm<+\infty\bigg\},
\qquad\mathcal L^\infty(\mm)\coloneqq\bigg\{f\in\mathcal L^0(\Sigma)\;\bigg|\;\sup_\X|f|<+\infty\bigg\}.
\]
Moreover, we consider the quotient spaces \(L^p(\mm)\coloneqq\pi_\mm\big(\mathcal L^p(\mm)\big)\) and
\(L^\infty(\mm)\coloneqq\pi_\mm\big(\mathcal L^\infty(\mm)\big)\), which are Banach spaces if endowed with the usual pointwise operations and with the norms
\[
\|f\|_{L^p(\mm)}\coloneqq\bigg(\int|f|^p\,\d\mm\bigg)^{1/p},\qquad\|f\|_{L^\infty(\mm)}\coloneqq\underset{\X}{\rm ess\,sup}\,|f|,
\]
respectively. Recall also that \(L^q(\mm)\) is isomorphic as a Banach space to the dual of \(L^p(\mm)\).
\medskip

It is worth recalling that, given an arbitrary \(\sigma\)-finite measure space \((\X,\Sigma,\mm)\) and any exponent \(p\in[1,\infty)\),
the Lebesgue space \(L^p(\mm)\) is not necessarily separable. In fact, it holds
\begin{equation}\label{eq:equiv_meas_sep}
(\X,\Sigma,\mm)\text{ is separable }\quad\Longleftrightarrow\quad\text{ }L^p(\mm)\text{ is separable, for every }p\in[1,\infty),
\end{equation}
where \((\X,\Sigma,\mm)\) is said to be \textbf{separable} provided there exists a countable family \(\mathcal C\subseteq\Sigma\)
for which the following property holds: given any set \(E\in\Sigma\) with \(\mm(E)<+\infty\) and \(\varepsilon>0\), there exists \(F\in\mathcal C\)
such that \(\mm(E\Delta F)<\varepsilon\). The equivalence stated in \eqref{eq:equiv_meas_sep} is well-known; it follows, for instance,
from \cite[Lemma 2.14]{DMLP21}. We also point out that if \((\X,\sfd)\) is a complete and separable metric space, \(\Sigma\) is the Borel
\(\sigma\)-algebra of \(\X\), and \(\mm\) is a boundedly-finite Borel measure on \(\X\), then \((\X,\Sigma,\mm)\) is a separable measure space.
\subsection{Banach spaces}
Let us begin by fixing some basic terminology about Banach spaces. Given a Banach space \(\B\),  we denote by \(\B'\) its (continuous) dual space.
Moreover, we denote by \(B_\B\) and \(\mathbb S_\B\) the \textbf{closed unit ball} and the \textbf{unit sphere} of \(\B\), respectively.
Namely, we set
\[
 B_\B\coloneqq\big\{v\in\B\;\big|\;\|v\|_\B\leq 1\big\},\qquad\mathbb S_\B\coloneqq\big\{v\in\B\;\big|\;\|v\|_\B=1\big\}.
\]
In this paper we are mostly concerned with Hilbert, uniformly convex, and reflexive spaces. We recall the notion of uniform
convexity, just to fix a notation for the modulus of convexity.
\begin{definition}[Uniform convexity]
Let \(\B\) be a Banach space. Let us define the \textbf{modulus of convexity} \(\delta_\B\colon(0,2]\to[0,1]\) of the space \(\B\) as follows:
given any \(\varepsilon\in(0,2]\), we set
\[
\delta_\B(\varepsilon)\coloneqq\inf\bigg\{1-\bigg\|\frac{v+w}{2}\bigg\|_\B\;\bigg|\;v,w\in\mathbb S_\B,\,\|v-w\|_\B\geq\varepsilon\bigg\}.
\]
Then we say that \(\B\) is \textbf{uniformly convex} if and only if \(\delta_\B(\varepsilon)>0\) holds for every \(\varepsilon\in(0,2]\).
\end{definition}
It is well-known that the following implications are verified:
\[
\B\text{ is Hilbert}\quad\Longrightarrow\quad\B\text{ is uniformly convex}\quad\Longrightarrow\quad\B\text{ is reflexive.}
\]
The following elementary observation will play a r\^{o}le during the proof of Theorem \ref{thm:unif_conv_bundles_mod}.
\begin{remark}\label{rmk:check_unif_conv_on_dense}{\rm
The uniform convexity condition can be checked on a dense set. Namely, given any dense subset \(D\) of \(\mathbb S_\B\), one has that
for every \(\varepsilon\in(0,2]\) it holds that
\[
\delta_\B(\varepsilon)=\inf\bigg\{1-\bigg\|\frac{v+w}{2}\bigg\|_\B\;\bigg|\;v,w\in D,\,\|v-w\|_\B>\varepsilon\bigg\}.
\]
The proof of this claim can be easily obtained via a standard approximation argument.
\fr}\end{remark}
Let us briefly recall the basics of integration theory in the sense of Bochner. Fix a \(\sigma\)-finite measure space \((\X,\Sigma,\mm)\)
and a Banach space \(\B\). Let \(v\colon\X\to\B\) be given. Then we say that:
\begin{itemize}
\item \(v\) is \textbf{weakly measurable} if \(\X\ni x\mapsto\langle\omega,v(x)\rangle\in\R\) is measurable for every \(\omega\in\B'\).
\item \(v\) is \textbf{essentially separably valued} provided there exists an \(\mm\)-null set \(N\in\Sigma\) such that the image
\(v(\X\setminus N)\subseteq\B\) is separable.
\item \(v\) is \textbf{strongly measurable} if it is weakly measurable and essentially separably valued.
\end{itemize}
It is well-known that a weakly measurable map is strongly measurable if and only if there exists a sequence \((v_n)_{n\in\N}\) of simple maps
\(v_n\colon\X\to\B\) and an \(\mm\)-null set \(N\in\Sigma\) such that \(\lim_n\|v_n(x)-v(x)\|_\B=0\) holds for every \(x\in\X\setminus N\).
Here, by a \textbf{simple map} we mean a map \(w\colon\X\to\B\) which can be written as \(w=\sum_{i=1}^k\1_{A_i}w_i\), where
\((A_i)_{i=1}^k\subseteq\Sigma\) are pairwise disjoint with \(\mm(A_i)<\infty\) and \((w_i)_{i=1}^k\subseteq\B\).
\medskip

We denote by \(\mathcal L^0(\mm;\B)\) the space of all strongly measurable maps from \(\X\) to \(\B\), while for any given exponent
\(p\in[1,\infty)\) we define \(\mathcal L^p(\mm;\B)\coloneqq\big\{v\in\mathcal L^0(\mm;\B)\,:\,\int\|v(\cdot)\|_\B^p\,\d\mm<+\infty\big\}\),
and \(\mathcal L^\infty(\mm;\B)\coloneqq\big\{v\in\mathcal L^0(\mm;\B)\,:\,\sup_\X\|v(\cdot)\|_\B<+\infty\big\}\). These definitions are
well-posed, since \(\|v(\cdot)\|_\B\) is measurable thanks to the strong measurability of \(v\) and the continuity of \(\|\cdot\|_\B\).
We introduce an equivalence relation on \(\mathcal L^0(\mm;\B)\): given two elements \(v,w\in\mathcal L^0(\mm;\B)\), we declare that \(v\sim w\)
if and only if \(v(x)=w(x)\) for \(\mm\)-a.e.\ \(x\in\X\). Then we define
\[
L^0(\mm;\B)\coloneqq\mathcal L^0(\mm;\B)/\sim,
\]
while \(\pi_\mm\colon\mathcal L^0(\mm;\B)\to L^0(\mm;\B)\) stands for the projection map on the quotient. Moreover, we set
\(L^p(\mm;\B)\coloneqq\pi_\mm\big(\mathcal L^p(\mm;\B)\big)\) for every \(p\in[1,\infty]\). The linear space \(L^p(\mm;\B)\)
is a Banach space if endowed with the norm \(\|v\|_{L^p(\mm;\B)}\coloneqq\big\|\|v(\cdot)\|_\B\big\|_{L^p(\mm)}\). The spaces \(L^p(\mm;\B)\)
are called the \textbf{Lebesgue--Bochner spaces}. Note also that \(\mathcal L^p(\mm;\R)=\mathcal L^p(\mm)\) and \(L^p(\mm;\R)=L^p(\mm)\).
\medskip

Finally, we recall that a given strongly measurable map \(v\colon\X\to\B\) is said to be \textbf{Bochner integrable} on a set \(E\in\Sigma\)
provided \(\1_E\cdot v\in\mathcal L^1(\mm;\B)\). In this case, it is well-known that there exists a sequence \((v_n)_{n\in\N}\) of simple
maps \(v_n\colon\X\to\B\) such that \(\lim_n\int_E\|v_n(\cdot)-v(\cdot)\|_\B\,\d\mm=0\). In particular, it holds that the limit
\(\int_E v\,\d\mm\coloneqq\lim_n\int_E v_n\,\d\mm\in\B\) exists, where for any simple map \(w=\sum_{i=1}^k\1_{A_i}w_i\)
we set \(\int_E w\,\d\mm\coloneqq\sum_{i=1}^k\mm(A_i\cap E)w_i\in\B\). The element \(\int_E v\,\d\mm\), which is independent
of the specific choice of \((v_n)_n\), is called the \textbf{Bochner integral} of \(v\) on \(E\).
\subsection{Banach bundles}
Aim of this section is to recall the notion of Banach bundle introduced in \cite{DMLP21} and its main properties.
Let us fix a measurable space \((\X,\Sigma)\) and a Banach space \(\B\). By \(\boldsymbol\varphi\colon\X\twoheadrightarrow\B\)
we denote a \textbf{multivalued map}, \emph{i.e.}, a map from \(\X\) to the power set of \(\B\). Following \cite{AliprantisBorder99},
we say that \(\boldsymbol\varphi\) is \textbf{weakly measurable} provided \(\big\{x\in\X\,:\,\boldsymbol\varphi(x)\cap U\neq\varnothing\big\}\in\Sigma\)
holds for every open set \(U\subseteq\B\). The following definition is taken from \cite[Definition 4.1]{DMLP21} (cf.\ also with
\cite[Definition 2.15]{GLDDE} for the case of a non-separable ambient space \(\B\)):
\begin{definition}[Banach bundle]
Let \((\X,\Sigma)\) be a measurable space and \(\B\) a Banach space. Then a given weakly measurable multivalued map \(\mathbf E\colon\X\twoheadrightarrow\B\)
is said to be a \textbf{Banach \(\B\)-bundle} on \(\X\) provided \(\mathbf E(x)\) is a closed linear subspace of \(\B\) for every \(x\in\X\).
\end{definition}
Let us also introduce the following subclasses of Banach bundles, which will be studied in details in Sections \ref{s:Hilb_and_UC} and \ref{s:reflex}.
\begin{definition}[Hilbert, uniformly convex, and reflexive bundles]
Let \((\X,\Sigma,\mm)\) be a measure space, \(\B\) a Banach space, and \(\mathbf E\) a Banach \(\B\)-bundle over \(\X\). Then we say that:
\begin{itemize}
\item[\(\rm i)\)] \(\mathbf E\) is \textbf{Hilbert} provided \(\mathbf E(x)\) is Hilbert for \(\mm\)-a.e.\ \(x\in\X\).
\item[\(\rm ii)\)] \(\mathbf E\) is \textbf{uniformly convex} provided \(\mathbf E(x)\) is uniformly convex for \(\mm\)-a.e.\ \(x\in\X\).
\item[\(\rm iii)\)] \(\mathbf E\) is \textbf{reflexive} provided \(\mathbf E(x)\) is reflexive for \(\mm\)-a.e.\ \(x\in\X\).
\end{itemize}
\end{definition}

By a \textbf{section} of a Banach \(\B\)-bundle \(\mathbf E\) over \(\X\) we mean a measurable selector of \(\mathbf E\), \emph{i.e.},
a strongly measurable map \(v\colon\X\to\B\) with \(v(x)\in\mathbf E(x)\) for all \(x\in\X\). We denote by \(\bar\Gamma_0(\mathbf E)\)
the family of all sections of \(\mathbf E\). We introduce an equivalence relation on \(\bar\Gamma_0(\mathbf E)\):
given \(v,w\in\bar\Gamma_0(\mathbf E)\), we declare that \(v\sim w\) if and only if \(v(x)=w(x)\) for \(\mm\)-a.e.\ \(x\in\X\).
We then define
\[
\Gamma_0(\mathbf E)\coloneqq\bar\Gamma_0(\mathbf E)/\sim,
\]
while \(\pi_\mm\colon\bar\Gamma_0(\mathbf E)\to\Gamma_0(\mathbf E)\) stands for the projection map on the quotient.
By analogy with the case of Lebesgue--Bochner spaces, for any given exponent \(p\in[1,\infty]\) we define
\[
\bar\Gamma_p(\mathbf E)\coloneqq\bigg\{v\in\bar\Gamma_0(\mathbf E)\;\bigg|\;\|v(\cdot)\|_\B\in\mathcal L^p(\mm)\bigg\},
\qquad\Gamma_p(\mathbf E)\coloneqq\pi_\mm\big(\bar\Gamma_p(\mathbf E)\big).
\]
The previous definitions are well-posed, since the function \(\X\ni x\mapsto\|v(x)\|_\B\in\R\) is measurable thanks to the
strong measurability of \(v\) and the continuity of \(\|\cdot\|_\B\). One can readily check that \(\Gamma_p(\mathbf E)\) is
a Banach space if endowed with the pointwise operations and with the norm
\[
\|v\|_{\Gamma_p(\mathbf E)}\coloneqq\big\|\|v(\cdot)\|_\B\big\|_{L^p(\mm)},\quad\text{ for every }v\in\Gamma_p(\mathbf E).
\]
This is de facto a generalisation of Lebesgue--Bochner spaces: calling \(\B\) the Banach \(\B\)-bundle whose fibers
are constantly equal to \(\B\), it holds \(\bar\Gamma_p(\B)=\mathcal L^p(\mm;\B)\) and \(\Gamma_p(\B)=L^p(\mm;\B)\).
\begin{remark}{\rm
Consistently with the case of Lebesgue spaces, the space of sections \(\Gamma_p(\mathbf E)\) of a given Banach \(\B\)-bundle \(\mathbf E\) over \(\X\)
needs not be separable, even if \(\B\) is separable, \(\mm\) is \(\sigma\)-finite, and \(p\in[1,\infty)\). In fact, under the assumption that
the ambient space \(\B\) is separable, it holds
\[
\Gamma_p(\mathbf E)\text{ is separable, for every }p\in[1,\infty)\quad\Longleftrightarrow\quad(\X,\Sigma,\mm|_G)\text{ is separable},
\]
where \(G\coloneqq\big\{x\in\X\,:\,\mathbf E(x)\neq\{0_\B\}\big\}\). We omit the proof, similar to the one of \eqref{eq:equiv_meas_sep}.
\fr}\end{remark}
Hereafter, we shall focus on \(\sigma\)-finite measure spaces and Banach \(\B\)-bundles \(\mathbf E\) over \(\X\), where the space \(\B\) is separable.
We will need the following result, taken from \cite[Proposition 4.4]{DMLP21}.
\begin{proposition}\label{prop:spanning_sects}
Let \((\X,\Sigma,\mm)\) be a \(\sigma\)-finite measure space. Let \(\B\) be a separable Banach space and let \(\mathbf E\) be a Banach
\(\B\)-bundle over \(\X\). Let \(p\in[1,\infty)\) be given. Then there exists a countable \(\mathbb Q\)-linear subspace \(\mathcal C\)
of \(\bar \Gamma_p(\mathbf E)\) such that \(\mathbf E(x)={\rm cl}_\B\big\{v(x)\,:\,v\in\mathcal C\big\}\) for every \(x\in\X\).
\end{proposition}
In the sequel, it will be convenient to apply the Lebesgue Differentiation Theorem to the sections of a given Banach bundle.
In the particular case where the base space \(\X\) is a doubling metric measure space, this is classical result. However, we
are concerned with a much more general base space \(\X\), thus we need a generalised form of Lebesgue Differentiation Theorem,
obtained in \cite{GLDDE}. Before passing to its statement, we introduce some auxiliary terminology.
\medskip

Let \((\X,\Sigma,\mm)\) be a measure space. Then we say that a family \(\mathcal I\subseteq\Sigma\) is a \textbf{differentiation basis}
on \((\X,\Sigma,\mm)\) provided \(0<\mm(I)<+\infty\) for every \(I\in\mathcal I\), the set \(\mathcal I_x\coloneqq\{I\in\mathcal I\,:\,x\in I\}\)
is non-empty for \(\mm\)-a.e.\ \(x\in\X\), and each \(\mathcal I_x\) is \textbf{directed by downward inclusion}, meaning that for
any \(I,I'\in\mathcal I_x\) there exists \(J\in\mathcal I_x\) such that \(J\subseteq I\cap I'\). Every differentiation basis induces
a notion of limit, as follows. Fix a metric space \((\Y,\sfd)\) and a mapping \(\Phi\colon\X\to\Y\). Then we declare that the
\textbf{\(\mathcal I\)-limit} of \(\Phi\) at a point \(x\in\X\) exists and coincides with \(y\in\Y\) if for any \(\varepsilon>0\) there
exists \(I\in\mathcal I_x\) such that \(\sfd\big(\Phi(J),y\big)<\varepsilon\) holds for every \(J\in\mathcal I_x\) with \(J\subseteq I\).
In this case, we write \(\lim_{I\Rightarrow x}\Phi(I)=y\) for brevity. When \(\Y=\R\), we also define the \textbf{\(\mathcal I\)-limit superior}
\(\lims_{I\Rightarrow x}\Phi(I)\in\R\cup\{\pm\infty\}\) and the \textbf{\(\mathcal I\)-limit inferior} \(\limi_{I\Rightarrow x}\Phi(I)\in\R\cup\{\pm\infty\}\) as
\[\begin{split}
\lims_{I\Rightarrow x}\Phi(I)&\coloneqq\inf\Big\{\sigma\in\R\;\Big|\;\exists\,I\in\mathcal I_x:\;\Phi(J)\leq\sigma,
\text{ for every }J\in\mathcal I_x\text{ with }J\subseteq I\Big\},\\
\limi_{I\Rightarrow x}\Phi(I)&\coloneqq\sup\Big\{\sigma\in\R\;\Big|\;\exists\,I\in\mathcal I_x:\;\Phi(J)\geq\sigma,
\text{ for every }J\in\mathcal I_x\text{ with }J\subseteq I\Big\},
\end{split}\]
respectively. Observe that \(\lim_{I\Rightarrow x}\Phi(I)\) exists if and only if \(\limi_{I\Rightarrow x}\Phi(I)=\lims_{I\Rightarrow x}\Phi(I)\in\R\).
In this case, it also holds that \(\lim_{I\Rightarrow x}\Phi(I)=\limi_{I\Rightarrow x}\Phi(I)=\lims_{I\Rightarrow x}\Phi(I)\). Finally, we recall that a
strongly measurable map \(v\colon\X\to\B\) is said to be \textbf{locally Bochner integrable with respect to \(\mathcal I\)} if for \(\mm\)-a.e.\ \(x\in\X\)
there exists \(I\in\mathcal I_x\) such that \(v\) is Bochner integrable on \(I\).
\medskip

We are in a position to state the following theorem, which was proved in \cite[Theorem 3.5]{GLDDE}.
\begin{theorem}[Lebesgue Differentiation Theorem for sections]\label{thm:Leb_diff}
Let \((\X,\Sigma,\mm)\) be a complete, \(\sigma\)-finite measure space, \(\B\) a Banach space, \(\mathbf E\) a Banach \(\B\)-bundle over \(\X\).
Then there exists a differentiation basis \(\mathcal I\) on \((\X,\Sigma,\mm)\) such that the following property holds. Given \(p\in[1,\infty]\)
and \(v\in\Gamma_p(\mathbf E)\), it holds that the map \(v\) is locally Bochner integrable with respect to \(\mathcal I\) and that
\begin{equation}\label{eq:Lebesgue_pt}
\exists\,\hat v(x)\coloneqq\lim_{I\Rightarrow x}\fint_I v\,\d\mm\in\mathbf E(x)\;\;\text{and}\;\;
\lim_{I\Rightarrow x}\fint_I\big\|v(\cdot)-\hat v(x)\big\|_\B\,\d\mm=0,\quad\text{ for }\mm\text{-a.e.\ }x\in\X.
\end{equation}
Setting \(\hat v(x)\coloneqq 0_{\mathbf E(x)}\) elsewhere, we have that \(\hat v\in\bar\Gamma_p(\mathbf E)\)
and \(\hat v(x)=v(x)\) for \(\mm\)-a.e.\ \(x\in\X\).
\end{theorem}
We say that \(\hat v\) is the \textbf{precise representative} of \(v\). Every \(x\in\X\) where \eqref{eq:Lebesgue_pt} holds
is said to be a \textbf{Lebesgue point} of \(v\). We denote by \({\rm Leb}(v)\in\Sigma\) the set of all Lebesgue points of \(v\).
\medskip

In fact, a stronger property holds: as proven in \cite[Lemma 3.2]{GLDDE}, on any complete, \(\sigma\)-finite measure space it is possible to
find a differentiation basis \(\mathcal I\) such that every \(L^p\)-section is `approximately continuous' with respect to \(\mathcal I\).
The precise statement reads as follows.
\begin{theorem}\label{thm:approx_cont}
Let \((\X,\Sigma,\mm)\) be a complete, \(\sigma\)-finite measure space, \(\B\) a Banach space, \(\mathbf E\) a Banach \(\B\)-bundle over \(\X\).
Then there exists a differentiation basis \(\mathcal I\) on \((\X,\Sigma,\mm)\) such that the following property holds. Given any \(p\in[1,\infty]\) and
\(v\in\Gamma_p(\mathbf E)\), it is possible to find a representative \(\bar v\in\bar\Gamma_p(\mathbf E)\) of \(v\) such that for \(\mm\)-a.e.\ point
\(x\in\X\) it holds that
\begin{equation}\label{eq:approx_cont}
\forall\varepsilon>0,\quad\exists\,I\in\mathcal I_x:\quad\big\|\bar v(y)-\bar v(x)\big\|_\B<\varepsilon,\quad\text{for every }y\in I.
\end{equation}
By a point of \textbf{approximate continuity} of \(\bar v\) we mean a point \(x\in\X\) for which \eqref{eq:approx_cont} holds.
\end{theorem}
Note that any \(x\in\X\) of approximate continuity of \(\bar v\) is a Lebesgue point of \(v\) and \(\hat v(x)=\bar v(x)\).
\subsection{Banach modules}
In this section we recall the basics of the theory of Banach modules. We begin by introducing
the notion of \(L^p\)-normed \(L^\infty\)-module proposed by N.\ Gigli in \cite{Gigli14}.
\begin{definition}[\(L^p\)-normed \(L^\infty\)-module]
Let \((\X,\Sigma,\mm)\) be a \(\sigma\)-finite measure space and let \(p\in(1,\infty)\) be a given exponent. Let \(\mathscr M\) be a module over
\(L^\infty(\mm)\). Then a map \(|\cdot|\colon\mathscr M\to L^p(\mm)\) is said to be an \textbf{\(L^p(\mm)\)-pointwise norm} on \(\mathscr M\)
provided it verifies the following conditions:
\begin{subequations}\begin{align}
\label{eq:ptwse_norm1}
|v|\geq 0&,\quad\text{ for every }v\in\mathscr M,\text{ with equality if and only if }v=0,\\
\label{eq:ptwse_norm2}
|v+w|\leq|v|+|w|&,\quad\text{ for every }v,w\in\mathscr M,\\
\label{eq:ptwse_norm3}
|f\cdot v|=|f||v|&,\quad\text{ for every }f\in L^\infty(\mm)\text{ and }v\in\mathscr M,
\end{align}\end{subequations}
where equalities and inequalities are intended in the \(\mm\)-a.e.\ sense. The couple \((\mathscr M,|\cdot|)\) is called an
\textbf{\(L^p(\mm)\)-normed \(L^\infty(\mm)\)-module}. We endow \(\mathscr M\) with the norm \(\|\cdot\|_{\mathscr M}\), given by
\[
\|v\|_{\mathscr M}\coloneqq\big\||v|\big\|_{L^p(\mm)},\quad\text{ for every }v\in\mathscr M.
\]
When \(\|\cdot\|_{\mathscr M}\) is complete, we say that \((\mathscr M,|\cdot|)\) is an \textbf{\(L^p(\mm)\)-Banach \(L^\infty(\mm)\)-module}.
\end{definition}
A prototypical example of \(L^p(\mm)\)-Banach \(L^\infty(\mm)\)-module is the space of \(L^p(\mm)\)-sections of a Banach \(\B\)-bundle
\(\mathbf E\) over \(\X\), the \(L^p(\mm)\)-pointwise norm on \(\Gamma_p(\mathbf E)\) being \(|v|\coloneqq\|v(\cdot)\|_{\mathbf E(\cdot)}\).
\medskip

An important class of \(L^2(\mm)\)-Banach \(L^\infty(\mm)\)-modules is the one of Hilbert modules. Following \cite[Definition 1.2.20]{Gigli14},
we say that an \(L^2(\mm)\)-Banach \(L^\infty(\mm)\)-module \(\mathscr H\) is \textbf{Hilbert} provided it is Hilbert when viewed as a Banach space.
It is shown in \cite[Proposition 1.2.21]{Gigli14} that \(\mathscr H\) is a Hilbert module if and only if the pointwise parallelogram identity holds:
\[
|v+w|^2+|v-w|^2=2|v|^2+2|w|^2\,\;\;\mm\text{-a.e.,}\quad\text{ for every }v,w\in\mathscr H.
\]

An operator between two \(L^p(\mm)\)-normed \(L^\infty(\mm)\)-modules is called a \textbf{homomorphism of \(L^p(\mm)\)-normed \(L^\infty(\mm)\)-modules}
provided it is \(L^\infty(\mm)\)-linear and continuous. The \textbf{dual} of an \(L^p(\mm)\)-normed \(L^\infty(\mm)\)-module \(\mathscr M\) is
given by the space \(\mathscr M^*\) of all \(L^\infty(\mm)\)-linear and continuous maps from \(\mathscr M\) to \(L^1(\mm)\).
It holds that \(\mathscr M^*\) is an \(L^q(\mm)\)-Banach \(L^\infty(\mm)\)-module if endowed with the \(L^q(\mm)\)-pointwise
norm operator \(|\cdot|\colon\mathscr M^*\to L^q(\mm)\), which is defined as
\[
|\omega|\coloneqq\bigvee\big\{\omega(v)\;\big|\;v\in\mathscr M,\,|v|\leq 1\text{ holds }\mm\text{-a.e.}\big\},\quad\text{ for every }\omega\in\mathscr M^*.
\]
Furthermore, we denote by \(J_{\mathscr M}\colon\mathscr M\to\mathscr M^{**}\) the \textbf{James' embedding} of \(\mathscr M\) into its bidual, \emph{i.e.},
the unique homomorphism of \(L^p(\mm)\)-Banach \(L^\infty(\mm)\)-modules satisfying
\begin{equation}\label{eq:def_James}
\langle J_{\mathscr M}(v),\omega\rangle=\langle\omega,v\rangle,\quad\text{ for every }v\in\mathscr M\text{ and }\omega\in\mathscr M^*.
\end{equation}
We have that \(|J_{\mathscr M}(v)|=|v|\) holds \(\mm\)-a.e.\ for every \(v\in\mathscr M\). Then the space \(\mathscr M\) is said to be
\textbf{reflexive (as a module)} provided \(J_{\mathscr M}\) is surjective (and thus an isomorphism). According to \cite[Corollary 1.2.18]{Gigli14},
\(\mathscr M\) is reflexive if and only if it is reflexive as a Banach space.
\medskip

Let us also recall the notion of \textbf{adjoint operator}: given a homomorphism \(\varphi\colon\mathscr M\to\mathscr N\) between two
\(L^p(\mm)\)-Banach \(L^\infty(\mm)\)-modules \(\mathscr M\) and \(\mathscr N\), we denote by \(\varphi^{\rm ad}\colon\mathscr N^*\to\mathscr M^*\)
the unique homomorphism of \(L^q(\mm)\)-Banach \(L^\infty(\mm)\)-modules such that
\begin{equation}\label{eq:def_ad}
\langle\varphi^{\rm ad}(\omega),v\rangle=\langle\omega,\varphi(v)\rangle,\quad\text{ for every }\omega\in\mathscr N^*\text{ and }v\in\mathscr M.
\end{equation}
It holds that \(\varphi^{\rm ad}\) is an isomorphism if and only if \(\varphi\) is an isomorphism.
\medskip

In the theory of Banach modules, it is often convenient to drop the integrability assumption:
\begin{definition}[\(L^0\)-normed \(L^0\)-module]
Let \((\X,\Sigma,\mm)\) be a \(\sigma\)-finite measure space. Let \(\mathscr M\) be a module over \(L^0(\mm)\). Then a map \(|\cdot|\colon\mathscr M\to L^0(\mm)\)
is said to be an \textbf{\(L^0(\mm)\)-pointwise norm} on \(\mathscr M\) provided it verifies \eqref{eq:ptwse_norm1}, \eqref{eq:ptwse_norm2}, \eqref{eq:ptwse_norm3},
but replacing \(L^\infty(\mm)\) with \(L^0(\mm)\) in \eqref{eq:ptwse_norm3}. The couple \((\mathscr M,|\cdot|)\) is called an
\textbf{\(L^0(\mm)\)-normed \(L^0(\mm)\)-module}, or a \textbf{random normed module over \(\R\) with base \((\X,\Sigma,\mm)\)}
in the case where \(\mm\) is a probability measure. We also endow \(\mathscr M\) with the distance \(\sfd_{\mathscr M}\), given by
\begin{equation}\label{eq:def_L0_dist}
\sfd_{\mathscr M}(v,w)\coloneqq\int\min\{|v-w|,1\}\,\d\mm',\quad\text{ for every }v\in\mathscr M,
\end{equation}
where \(\mm'\) is any given probability measure on \(\Sigma\) satisfying \(\mm\ll\mm'\ll\mm\). When \(\sfd_{\mathscr M}\) is complete,
we say that \((\mathscr M,|\cdot|)\) is an \textbf{\(L^0(\mm)\)-Banach \(L^0(\mm)\)-module}.
\end{definition}
A key example of \(L^0(\mm)\)-Banach \(L^0(\mm)\)-module is the space \(\Gamma_0(\mathbf E)\), where \(\mathbf E\) is a Banach \(\B\)-bundle
over \(\X\), and as an \(L^0(\mm)\)-pointwise norm on \(\Gamma_0(\mathbf E)\) we consider \(|v|\coloneqq\|v(\cdot)\|_{\mathbf E(\cdot)}\).
\medskip

It is worth pointing out that a random normed module is complete with respect to the distance introduced in \eqref{eq:def_L0_dist}
if and only if it is complete in the sense of \cite{GZRNM,GZRNM2}, \emph{i.e.}, with respect to the so-called \((\epsilon,\lambda)\)-topology.
Indeed, both the topology induced by the \(L^0\)-distance and the \((\epsilon,\lambda)\)-topology coincide with the one of
`convergence in measure', cf.\ \cite{GP20} and \cite{Guo99}.
\medskip

An operator between two \(L^0(\mm)\)-normed \(L^0(\mm)\)-modules is called a \textbf{homomorphism of \(L^0(\mm)\)-normed \(L^0(\mm)\)-modules}
provided it is \(L^0(\mm)\)-linear and continuous.
\medskip

The relation between \(L^p(\mm)\)-Banach \(L^\infty(\mm)\)-modules and \(L^0(\mm)\)-Banach \(L^0(\mm)\)-modules
can be expressed by the following result, which is taken from \cite[Theorem/Definition 2.7]{Gigli17}.
\begin{proposition}[\(L^0\)-completion]
Let \((\X,\Sigma,\mm)\) be a \(\sigma\)-finite measure space. Let \(\mathscr M\) be an \(L^p(\mm)\)-Banach \(L^\infty(\mm)\)-module,
for some exponent \(p\in(1,\infty)\). Then there exists a unique couple \((\mathscr M^0,\iota)\), where \(\mathscr M^0\) is an
\(L^0(\mm)\)-Banach \(L^0(\mm)\)-module, while \(\iota\colon\mathscr M\to\mathscr M^0\) is a linear operator which preserves the
pointwise norm and has dense image. Uniqueness is intended up to unique isomorphism: given another couple \((\mathscr N^0,\iota')\) having
the same properties, there exists a unique isomorphism of \(L^0(\mm)\)-Banach \(L^0(\mm)\)-modules \(\Phi\colon\mathscr M^0\to\mathscr N^0\)
such that \(\iota'=\Phi\circ\iota\). The space \(\mathscr M^0\) is called the \textbf{\(L^0(\mm)\)-completion} of \(\mathscr M\).
\end{proposition}
\section{Main results}
\subsection{Hilbertian and uniformly convex bundles/modules}\label{s:Hilb_and_UC}
In this section, we prove that a given separable Banach bundle is Hilbert (resp.\ uniformly convex) if and only if its
space of sections is Hilbert (resp.\ uniformly convex). Let us begin with Hilbert bundles/modules.
\begin{theorem}[Hilbert bundles/modules]\label{thm:Hilb_bundles_mod}
Let \((\X,\Sigma,\mm)\) be a \(\sigma\)-finite measure space. Let \(\B\) be a separable Banach space and let \(\mathbf E\)
be a Banach \(\B\)-bundle over \(\X\). Then \(\mathbf E\) is a Hilbert bundle if and only if \(\Gamma_2({\mathbf E})\)
is a Hilbert space. Necessity holds also when \(\B\) is non-separable.
\end{theorem}
\begin{proof}
Suppose \(\B\) is an arbitrary Banach space and \(\mathbf E\) is a Hilbert bundle. Fix \(v,w\in\Gamma_2(\mathbf E)\).
Then \(\|v(x)+w(x)\|_{\mathbf E(x)}^2+\|v(x)-w(x)\|_{\mathbf E(x)}^2=2\|v(x)\|_{\mathbf E(x)}^2+2\|w(x)\|_{\mathbf E(x)}^2\)
for \(\mm\)-a.e.\ \(x\in\X\). By integrating it over \(\X\) we obtain \(\|v+w\|_{\Gamma_2(\mathbf E)}^2+\|v-w\|_{\Gamma_2(\mathbf E)}^2
=2\|v\|_{\Gamma_2(\mathbf E)}^2+2\|w\|_{\Gamma_2(\mathbf E)}^2\), whence it follows that \(\Gamma_2(\mathbf E)\) is a Hilbert module.

Conversely, suppose \(\B\) is separable and \(\Gamma_2(\mathbf E)\) is a Hilbert module. Thanks to Proposition \ref{prop:spanning_sects},
we can find a \(\mathbb Q\)-linear space \((v_n)_{n\in\N}\subseteq\Gamma_2(\mathbf E)\) such that \(\{v_n(x)\,:\,n\in\N\}\) is dense in \(\mathbf E(x)\) for
\(\mm\)-a.e.\ \(x\in\X\). We argue by contradiction: suppose there exists \(P'\in\Sigma\) with \(\mm(P')>0\) such that \(\mathbf E(x)\)
is non-Hilbert for every \(x\in P'\). Hence, there must exist \(n,m\in\N\) and \(P\in\Sigma\), with \(P\subseteq P'\) and \(\mm(P)>0\),
such that for \(\mm\)-a.e.\ \(x\in P\) it holds
\[
\|v_n(x)+v_m(x)\|_{\mathbf E(x)}^2+\|v_n(x)-v_m(x)\|_{\mathbf E(x)}^2<2\|v_n(x)\|_{\mathbf E(x)}^2+2\|v_m(x)\|_{\mathbf E(x)}^2.
\]
By integrating the above inequality over \(P\), we conclude that
\[
\|\1_P\cdot v_n+\1_P\cdot v_m\|_{\Gamma_2(\mathbf E)}^2+\|\1_P\cdot v_n-\1_P\cdot v_m\|_{\Gamma_2(\mathbf E)}^2
<2\|\1_P\cdot v_n\|_{\Gamma_2(\mathbf E)}^2+2\|\1_P\cdot v_m\|_{\Gamma_2(\mathbf E)}^2,
\]
thus leading to a contradiction with the assumption that \(\Gamma_2(\mathbf E)\) is a Hilbert module.
\end{proof}
Next we aim at obtaining the analogue of Theorem \ref{thm:Hilb_bundles_mod}, but with the term `Hilbert' replaced by `uniformly convex'.
Its proof, which is more involved than the one for the Hilbertian case, requires some auxiliary notions and results. More precisely,
we have to work with an intermediate concept of `pointwise uniform convexity', which we introduce in Definition \ref{def:ptwse_unif_conv}.
It is a slight generalisation of the notion of random uniform convexity, proposed and studied by Guo--Zeng in \cite{GZRNM,GZRNM2}.
In Theorem \ref{thm:ptwse_unif_conv_equiv} we will extend their main result to \(\sigma\)-finite measures.
\begin{definition}[Pointwise uniform convexity]\label{def:ptwse_unif_conv}
Let \((\X,\Sigma,\mm)\) be a \(\sigma\)-finite measure space and let \(\mathscr M\) be an \(L^p(\mm)\)-Banach \(L^\infty(\mm)\)-module
for some exponent \(p\in(1,\infty)\). Let us define the \textbf{pointwise modulus of convexity} \(\delta_{\mathscr M}^{\rm pw}\colon
(0,2]\to L^\infty(\mm)\) associated with the space \(\mathscr M\) as
\[
\delta_{\mathscr M}^{\rm pw}(\varepsilon)\coloneqq\bigwedge\bigg\{1-\1_E\bigg|\frac{v+w}{2}\bigg|\;\bigg|\;
v,w\in\mathscr M,\,E\in\Sigma,\,|v|=|w|=1\text{ and }|v-w|>\varepsilon\text{ }\mm\text{-a.e.\ on }E\bigg\},
\]
for every \(\varepsilon\in(0,2]\). Then we say that \(\mathscr M\) is \textbf{pointwise uniformly convex} if and only if
\[
{\rm ess\,inf}\,\delta_{\mathscr M}^{\rm pw}(\varepsilon)>0,\quad\text{ for every }\varepsilon\in(0,2].
\]
\end{definition}
The very same definition of the function \(\delta_{\mathscr M}^{\rm pw}\) can be given also
in the case where \(\mathscr M\) is an \(L^0(\mm)\)-Banach \(L^0(\mm)\)-module. When \(\mm\) is a probability measure,
the notion of pointwise uniform convexity introduced in Definition \ref{def:ptwse_unif_conv} coincides with the one of
\emph{random uniform convexity} (see \cite[Definition 4.1]{GZRNM}). One can also easily verify that, given an \(L^p(\mm)\)-Banach
\(L^\infty(\mm)\)-module and called \(\mathscr M^0\) its \(L^0(\mm)\)-completion, it holds that
\(\delta_{\mathscr M^0}^{\rm pw}=\delta_{\mathscr M}^{\rm pw}\). In particular, given a Banach \(\B\)-bundle \(\mathbf E\) on \(\X\)
and two exponents \(p,p'\in(1,\infty)\), one has \(\delta_{\Gamma_p(\mathbf E)}^{\rm pw}=\delta_{\Gamma_{p'}(\mathbf E)}^{\rm pw}\).
\medskip

In the following result we prove the `easy implication' of Theorem \ref{thm:ptwse_unif_conv_equiv}.
The proof argument is strongly inspired by \cite[Theorem 4.3]{GZRNM}.
\begin{lemma}\label{lem:partial_pw_unif_convex}
Let \((\X,\Sigma,\mm)\) be a \(\sigma\)-finite measure space and let \(\mathscr M\) be an \(L^p(\mm)\)-Banach
\(L^\infty(\mm)\)-module for some exponent \(p\in(1,\infty)\). Then it holds that
\begin{equation}\label{eq:partial_pw_unif_convex}
\delta_{\mathscr M}(\varepsilon)\leq{\rm ess\,inf}\,\delta_{\mathscr M}^{\rm pw}(\varepsilon),
\quad\text{ for every }\varepsilon\in(0,2].
\end{equation}
\end{lemma}
\begin{proof}
Fix \(\varepsilon\in(0,2]\). Take also \(v,w\in\mathscr M\) and \(E\in\Sigma\) with \(\mm(E)>0\) such that \(|v|=|w|=1\)
and \(|v-w|>\varepsilon\) hold \(\mm\)-a.e.\ on \(E\). Given any set \(F\in\Sigma\) such that \(F\subseteq E\) and
\(0<\mm(F)<+\infty\), we define \(v_F,w_F\in\mathscr M\) as \(v_F\coloneqq\mm(F)^{-1/p}\1_F\cdot v\) and
\(w_F\coloneqq\mm(F)^{-1/p}\1_F\cdot w\). Observe that
\[
\|v_F\|_{\mathscr M}=\bigg(\int|v_F|^p\,\d\mm\bigg)^{\frac{1}{p}}=\bigg(\fint_F|v|^p\,\d\mm\bigg)^{\frac{1}{p}}=1.
\]
By similar computations, one can show that \(\|w_F\|_{\mathscr M}=1\) and \(\|v_F-w_F\|_{\mathscr M}\geq\varepsilon\).
Consequently, we have that \(1-\|v_F+w_F\|_{\mathscr M}/2\geq\delta_{\mathscr M}(\varepsilon)\), which can be equivalently rewritten as
\[
\fint_F\bigg|\frac{v+w}{2}\bigg|^p\,\d\mm=\int\bigg|\frac{v_F+w_F}{2}\bigg|^p\,\d\mm=
\frac{\|v_F+w_F\|_{\mathscr M}^p}{2^p}\leq\big(1-\delta_{\mathscr M}(\varepsilon)\big)^p.
\]
Thanks to the arbitrariness of \(F\), we deduce that \(\big|\frac{v+w}{2}\big|^p\leq\big(1-\delta_{\mathscr M}(\varepsilon)\big)^p\)
holds \(\mm\)-a.e.\ on \(E\). This implies that \(\delta_{\mathscr M}^{\rm pw}(\varepsilon)\geq\delta_{\mathscr M}(\varepsilon)\)
holds \(\mm\)-a.e.\ on \(\X\), which is equivalent to \eqref{eq:partial_pw_unif_convex}.
\end{proof}
The following theorem, which states that the pointwise uniform convexity of any given Banach module is equivalent to its
uniform convexity as a Banach space, is an almost direct consequence of a beautiful result obtained by Guo--Zeng in \cite{GZRNM,GZRNM2}.
\begin{theorem}\label{thm:ptwse_unif_conv_equiv}
Let \((\X,\Sigma,\mm)\) be a \(\sigma\)-finite measure space and let \(\mathscr M\) be an \(L^p(\mm)\)-Banach \(L^\infty(\mm)\)-module for
some exponent \(p\in(1,\infty)\). Then \(\mathscr M\) is pointwise uniformly convex if and only if it is uniformly convex as a Banach space.
In this case, for every \(\varepsilon\in(0,2]\) it holds that
\begin{equation}\label{eq:bounds_mod_convex}
\phi_p\big(\varepsilon,{\rm ess\,inf}\,\delta_{\mathscr M}^{\rm pw}(\varepsilon/5)\big)
\leq\delta_{\mathscr M}(\varepsilon)\leq{\rm ess\,inf}\,\delta_{\mathscr M}^{\rm pw}(\varepsilon),
\end{equation}
for some \(\phi_p\colon(0,2]\times(0,1]\to(0,1]\) that is continuous and non-decreasing in each variable.
\end{theorem}
\begin{proof}
We first deal with the case where \(\mm\) is a probability measure. Call \(\mathscr M^0\) the \(L^0(\mm)\)-completion
of \(\mathscr M\), which is a complete random normed module over \(\R\) with base \((\X,\Sigma,\mm)\). The fact that \(\mathscr M\) is pointwise
uniformly convex if and only if it is uniformly convex in the sense of Banach spaces was proved in \cite[Theorem 1.3]{GZRNM2} after
\cite[Theorems 4.3 and 4.4]{GZRNM}; the sufficiency part and the upper bound in \eqref{eq:bounds_mod_convex} follow also from
Lemma \ref{lem:partial_pw_unif_convex}. By inspecting the proof of \cite[Proposition 4.5]{GZRNM}, one can see that there exists
a function \(\phi_p\colon(0,2]\times(0,1]\to(0,1]\) that is continuous and non-decreasing in each variable, and that for any
\(\varepsilon\in(0,2]\) satisfies
\begin{equation}\label{eq:bounds_mod_convex_stronger}
\phi_p\big(\varepsilon,{\rm ess\,inf}\,\delta_{\mathscr M}^{\rm pw}(\varepsilon/4)\big)\leq\delta_{\mathscr M}(\varepsilon),\quad\text{ when }\mm(\X)=1.
\end{equation}
Next we explain how to drop the finiteness assumption on \(\mm\). For an arbitrary \(\sigma\)-finite measure \(\mm\), the sufficiency
part of the statement and the upper bound in \eqref{eq:bounds_mod_convex} follow, as before, directly from Lemma \ref{lem:partial_pw_unif_convex}.
Concerning the lower bound in \eqref{eq:bounds_mod_convex}, which in turn implies the necessity part of the statement, we argue as follows.
First, we set \(f(\varepsilon')\coloneqq{\rm ess\,inf}\,\delta_{\mathscr M}^{\rm pw}(\varepsilon')\) for every \(\varepsilon'\in(0,2]\).
Notice that \(f\) is non-decreasing by construction.
Now fix any \(\varepsilon\in(0,2]\). We may suppose that \(f(\varepsilon/5)>0\), otherwise there is nothing to prove.
Pick some sequence \((\varepsilon_k)_{k\in\N}\subseteq(4\varepsilon/5,\varepsilon)\) such that \(\varepsilon_k\to\varepsilon\).
Fix an increasing sequence of sets \((E_n)_{n\in\N}\subseteq\Sigma\) such that \(0<\mm(E_n)<+\infty\) for every \(n\in\N\) and
\(\bigcup_{n\in\N}E_n=\X\). The localised space \(\mathscr M|_{E_n}\) can be regarded as an \(L^p(\mm_n)\)-Banach \(L^\infty(\mm_n)\)-module,
where \(\mm_n\) stands for the normalised measure \(\mm(E_n)^{-1}\mm|_{E_n}\). Given that, trivially, \(\delta_{\mathscr M|_{E_n}}^{\rm pw}(\varepsilon')
\geq\delta_{\mathscr M}^{\rm pw}(\varepsilon')\) holds \(\mm_n\)-a.e.\ for every \(\varepsilon'\in(0,2]\), we deduce from \eqref{eq:bounds_mod_convex_stronger}
that \(\delta_{\mathscr M|_{E_n}}(\varepsilon')\geq\phi_p\big(\varepsilon',f(\varepsilon'/4)\big)\) for every \(\varepsilon'\in(0,2]\).
Now fix \(k\in\N\) and \(v,w\in\mathscr M\) such that \(\|v\|_{\mathscr M}=\|w\|_{\mathscr M}=1\) and \(\|v-w\|_{\mathscr M}\geq\varepsilon_k\).
We define \(v_n,w_n\in\mathscr M|_{E_n}\) as
\[
v_n\coloneqq\bigg(\frac{\mm(E_n)}{\int_{E_n}|v|^p\,\d\mm}\bigg)^{\frac{1}{p}}\1_{E_n}\cdot v,\qquad
w_n\coloneqq\bigg(\frac{\mm(E_n)}{\int_{E_n}|w|^p\,\d\mm}\bigg)^{\frac{1}{p}}\1_{E_n}\cdot w,
\]
for every \(n\in\N\) sufficiently big (so that \(\int_{E_n}|v|^p\,\d\mm\) and \(\int_{E_n}|w|^p\,\d\mm\) are non-zero). Hence,
recalling that we are viewing \(\mathscr M|_{E_n}\) as a Banach module over \((\X,\Sigma,\mm_n)\), we have that
\[
\|v_n\|_{\mathscr M|_{E_n}}=\bigg(\int|v_n|^p\,\d\mm_n\bigg)^{\frac{1}{p}}=
\bigg(\frac{\mm(E_n)}{\int_{E_n}|v|^p\,\d\mm}\int_{E_n}|v|^p\,\d\mm_n\bigg)^{\frac{1}{p}}=1
\]
and similarly \(\|w_n\|_{\mathscr M|_{E_n}}=1\). Moreover, by dominated convergence theorem we deduce that
\[
\tilde\varepsilon_n\coloneqq\|v_n-w_n\|_{\mathscr M|_{E_n}}=\bigg(\int_{E_n}\bigg|\frac{v}{\big(\int_{E_n}|v|^p\,\d\mm\big)^{1/p}}
-\frac{w}{\big(\int_{E_n}|w|^p\,\d\mm\big)^{1/p}}\bigg|^p\,\d\mm\bigg)^{\frac{1}{p}}\overset{n}\rightarrow\|v-w\|_{\mathscr M}.
\]
Analogous computations yield \(\big\|\frac{v_n+w_n}{2}\big\|_{\mathscr M|_{E_n}}\to\big\|\frac{v+w}{2}\big\|_{\mathscr M}\) as \(n\to\infty\).
We thus obtain that
\[\begin{split}
1-\bigg\|\frac{v+w}{2}\bigg\|_{\mathscr M}&=\lim_{n\to\infty}\bigg(1-\bigg\|\frac{v_n+w_n}{2}\bigg\|_{\mathscr M|_{E_n}}\bigg)\geq
\lims_{n\to\infty}\delta_{\mathscr M|_{E_n}}(\tilde\varepsilon_n)\geq\lims_{n\to\infty}\phi_p\big(\tilde\varepsilon_n,f(\tilde\varepsilon_n/4)\big)\\
&\geq\lim_{n\to\infty}\phi_p\big(\tilde\varepsilon_n,f(\varepsilon/5)\big)\overset{\star}=
\phi_p\big(\|v-w\|_{\mathscr M},f(\varepsilon/5)\big)\geq\phi_p\big(\varepsilon_k,f(\varepsilon/5)\big),
\end{split}\]
where in the passage from the first to the second line we used the fact that \(\tilde\varepsilon_n/4>\varepsilon/5\) holds for all \(n\in\N\)
sufficiently large, while for the starred equality we used the continuity of the function \(\varepsilon'\mapsto\phi_p\big(\varepsilon',f(\varepsilon/5)\big)\).
This implies \(\delta_{\mathscr M}(\varepsilon)\geq\delta_{\mathscr M}(\varepsilon_k)\geq\phi_p\big(\varepsilon_k,f(\varepsilon/5)\big)\)
for every \(k\in\N\). By letting \(k\to\infty\),
we can finally conclude that \(\delta_{\mathscr M}(\varepsilon)\geq\phi_p\big(\varepsilon,f(\varepsilon/5)\big)\), as desired.
\end{proof}
Before stating the main result of this section, we point out the following technical fact.
\begin{remark}{\rm
Let \((\X,\Sigma,\mm)\) be a \(\sigma\)-finite measure space, \(\B\) a separable Banach space, \(\mathbf E\) a Banach \(\B\)-bundle over \(\X\).
Then \(\X\ni x\mapsto\delta_{\mathbf E(x)}(\varepsilon)\in[0,1]\) is measurable for every \(\varepsilon\in(0,2]\).

Indeed, given any \(\mathcal C\) as in Proposition \ref{prop:spanning_sects}, we have that
\(\big\{v(x)/\|v(x)\|_\B\,:\,v\in\mathcal C,\,v(x)\neq 0\big\}\) is dense in \(\mathbb S_{\mathbf E(x)}\) for every \(x\in\X\),
thus Remark \ref{rmk:check_unif_conv_on_dense} ensures that \(\delta_{\mathbf E(x)}(\varepsilon)\) can be written as
\[
\inf\bigg\{1-\frac{1}{2}\bigg\|\frac{v(x)}{\|v(x)\|_\B}+\frac{w(x)}{\|w(x)\|_\B}\bigg\|_\B\;\bigg|\;
v,w\in\mathcal C,\,v(x),w(x)\neq 0,\,\bigg\|\frac{v(x)}{\|v(x)\|_\B}-\frac{w(x)}{\|w(x)\|_\B}\bigg\|_\B>\varepsilon\bigg\},
\]
for every \(x\in\X\). In particular, the function \(\delta_{\mathbf E(\cdot)}(\varepsilon)\), which can be expressed as
a countable infimum of measurable functions, is measurable. This yields the claim.
\fr}\end{remark}
Finally, we are in a position to prove the equivalence between the uniform convexity of a given separable Banach bundle and the uniform
convexity of its space of sections.
\begin{theorem}[Uniformly convex bundles/modules]\label{thm:unif_conv_bundles_mod}
Let \((\X,\Sigma,\mm)\) be a \(\sigma\)-finite measure space. Let \(\B\) be a separable Banach space and let \(\mathbf E\)
be a Banach \(\B\)-bundle over \(\X\). Then
\begin{equation}\label{eq:equiv_to_pw_2}
\delta_{\mathbf E(x)}(\varepsilon)=\delta_{\Gamma_p(\mathbf E)}^{\rm pw}(\varepsilon)(x),
\quad\text{ for every }\varepsilon\in(0,2]\text{ and }\mm\text{-a.e.\ }x\in\X.
\end{equation}
Moreover, the following two conditions are equivalent:
\begin{itemize}
\item[\(\rm i)\)] \(\mathbf E\) is a uniformly convex bundle and \({\rm ess\,inf}\,\delta_{\mathbf E(\cdot)}(\varepsilon)>0\)
for every \(\varepsilon>0\).
\item[\(\rm ii)\)] \(\Gamma_p({\mathbf E})\) is uniformly convex for every (or, equivalently, for some) exponent \(p\in(1,\infty)\).
\end{itemize}
More precisely, letting the function \(\phi_p\colon(0,2]\times(0,1]\to(0,1]\) be as in Theorem \ref{thm:ptwse_unif_conv_equiv}, it holds
\[
{\rm ess\,inf}\,\phi_p\big(\varepsilon,\delta_{\mathbf E(\cdot)}(\varepsilon/5)\big)\leq\delta_{\Gamma_p(\mathbf E)}(\varepsilon)
\leq{\rm ess\,inf}\,\delta_{\mathbf E(\cdot)}(\varepsilon),\quad\text{ for every }\varepsilon\in(0,2]\text{ and }p\in(1,\infty).
\]
\end{theorem}
\begin{proof}
Fix some \(\mathcal C\) as in Proposition \ref{prop:spanning_sects}.
For any \(v\in\mathcal C\), choose a representative \(\bar v\in\bar\Gamma_\infty(\mathbf E)\) of
\((\1_{\{|v|>0\}}|v|^{-1})\cdot v\) and define \(A_v\coloneqq\{|\bar v|>0\}=\{|\bar v|=1\}\in\Sigma\). We claim that
\begin{equation}\label{eq:equiv_to_pw_aux1}
\delta_{\Gamma_p(\mathbf E)}^{\rm pw}(\varepsilon)=\bigwedge\bigg\{1-\1_E\bigg|\frac{\bar v+\bar w}{2}\bigg|\;\bigg|\;
v,w\in\mathcal C,\,E\in\Sigma,\,E\subseteq A_v\cap A_w,\,|\bar v-\bar w|>\varepsilon\text{ in }E\bigg\}
\end{equation}
holds in the \(\mm\)-a.e.\ sense for any given \(\varepsilon\in(0,2]\). The inequality \(\leq\) is obvious, so let us focus on the converse one.
Pick any \(u,z\in\Gamma_p(\mathbf E)\) and \(E\in\Sigma\) such that \(|u|=|z|=1\) and \(|u-z|>\varepsilon\) \(\mm\)-a.e.\ on \(E\). Fix
\(\lambda\in(0,1)\). We aim at finding a partition \((E_k)_{k\in\N}\subseteq\Sigma\) of \(E\) up to \(\mm\)-null sets and vector fields
\(v_k,w_k\in\mathcal C\) with \(E_k\subseteq A_{v_k}\cap A_{w_k}\) such that \(|\bar v_k-\bar w_k|>\varepsilon\) on \(E_k\) and
\begin{equation}\label{eq:equiv_to_pw_aux2}
1-\1_{E_k}\bigg|\frac{\bar v_k+\bar w_k}{2}\bigg|\leq 1-\1_{E_k}\bigg|\frac{u+z}{2}\bigg|+\lambda,\quad\text{ in the }\mm\text{-a.e.\ sense.}
\end{equation}
First, set \(\tilde E_1\coloneqq\big\{x\in E\,:\,|u-z|(x)>\varepsilon+\lambda\big\}\) and
\(\tilde E_{i+1}\coloneqq\big\{x\in E\,:\,|u-z|(x)>\varepsilon+\frac{\lambda}{i+1}\big\}\setminus\tilde E_i\) for every \(i\in\N\).
Note that \((\tilde E_i)_{i\in\N}\subseteq\Sigma\) is a partition of \(E\) up to \(\mm\)-null sets. Given any \(i\in\N\),
it follows from Proposition \ref{prop:spanning_sects} that there exists a partition \((\tilde E_{i,j})_{j\in\N}\subseteq\Sigma\)
of the set \(\tilde E_i\) and sequences \((v_{i,j})_{j\in\N},(w_{i,j})_{j\in\N}\subseteq\mathcal C\) such that
\(|v_{i,j}-u|,|w_{i,j}-z|\leq\frac{\lambda}{4(i+1)}\) holds \(\mm\)-a.e.\ in \(\tilde E_{i,j}\). In particular,
\(|v_{i,j}|,|w_{i,j}|\geq 1-\frac{\lambda}{4(i+1)}\) holds \(\mm\)-a.e.\ on \(\tilde E_{i,j}\), thus up to removing an \(\mm\)-null set from
\(\tilde E_{i,j}\) we can assume \(\tilde E_{i,j}\subseteq A_{v_{i,j}}\cap A_{w_{i,j}}\). We have the following \(\mm\)-a.e.\ estimates on \(\tilde E_{i,j}\):
\begin{equation}\label{eq:equiv_to_pw_aux3}\begin{split}
|\bar v_{i,j}-u|&\leq\bigg|\frac{v_{i,j}}{|v_{i,j}|}-v_{i,j}\bigg|+|v_{i,j}-u|=\big|1-|v_{i,j}|\big|+|v_{i,j}-u|\leq\frac{\lambda}{2(i+1)},\\
|\bar w_{i,j}-z|&\leq\frac{\lambda}{2(i+1)}.
\end{split}\end{equation}
Consequently, we have that \(|\bar v_{i,j}-\bar w_{i,j}|\geq|u-z|-|\bar v_{i,j}-u|-|\bar w_{i,j}-z|
>\varepsilon+\frac{\lambda}{i+1}-\frac{\lambda}{i+1}=\varepsilon\) holds
\(\mm\)-a.e.\ on \(\tilde E_{i,j}\). Moreover, again in the \(\mm|_{\tilde E_{i,j}}\)-a.e.\ sense, we can estimate
\[
1-\bigg|\frac{\bar v_{i,j}+\bar w_{i,j}}{2}\bigg|\leq 1-\bigg|\frac{u+z}{2}\bigg|+\frac{|\bar v_{i,j}-u|+|\bar w_{i,j}-z|}{2}
\overset{\eqref{eq:equiv_to_pw_aux3}}\leq 1-\bigg|\frac{u+z}{2}\bigg|+\frac{\lambda}{2(i+1)}<1-\bigg|\frac{u+z}{2}\bigg|+\lambda.
\]
Therefore, relabeling the countable family \(\big\{(\tilde E_{i,j},v_{i,j},w_{i,j})\,:\,i,j\in\N\big\}\) as
\(\big((E_k,v_k,w_k)\big)_{k\in\N}\), we have obtained \eqref{eq:equiv_to_pw_aux2}.
All in all, the claim \eqref{eq:equiv_to_pw_aux1} is proved. Now observe that there exists a set \(N\in\Sigma\) with \(\mm(N)=0\)
such that \(\big\{\bar v(x)\,:\,v\in\mathcal C,\,\bar v(x)\neq 0\big\}\) is dense in \(\mathbb S_{\mathbf E(x)}\) for every point
\(x\in\X\setminus N\). Hence, Remark \ref{rmk:check_unif_conv_on_dense} ensures that for any \(\varepsilon\in(0,2]\) it holds that
\[
\delta_{\mathbf E(x)}(\varepsilon)=\inf\bigg\{1-\bigg\|\frac{\bar v(x)+\bar w(x)}{2}\bigg\|_{\mathbf E(x)}\;\bigg|\;
v,w\in\mathcal C,\,\bar v(x),\bar w(x)\neq 0,\,\|\bar v(x)-\bar w(x)\|_{\mathbf E(x)}>\varepsilon\bigg\}.
\]
By combining the previous identity with \eqref{eq:equiv_to_pw_aux1}, we conclude that \eqref{eq:equiv_to_pw_2} is verified.
In particular, it holds \({\rm ess\,inf}\,\delta_{\mathbf E(\cdot)}(\varepsilon)={\rm ess\,inf}\,\delta_{\Gamma_p(\mathbf E)}^{\rm pw}(\varepsilon)\).
By taking into account the properties of \(\phi_p\), we get
\[
{\rm ess\,inf}\,\phi_p\big(\varepsilon,\delta_{\mathbf E(\cdot)}(\varepsilon/5)\big)=
\phi_p\big(\varepsilon,{\rm ess\,inf}\,\delta_{\Gamma_p(\mathbf E)}^{\rm pw}(\varepsilon/5)\big),
\]
which is sufficient to conclude the proof of the statement thanks to Theorem \ref{thm:ptwse_unif_conv_equiv}.
\end{proof}
\subsection{Characterisation of the dual of a section space}\label{s:dual_sections}
Let \((\X,\Sigma,\mm)\) be a \(\sigma\)-finite measure space. Let \(\B\) be a Banach space and \(\mathbf E\) a Banach \(\B\)-bundle over \(\X\).
Then we define \(\bar\Gamma_0(\mathbf E'_{w^*})\) as the space of all maps \(\bar\omega\colon\X\to\bigsqcup_{x\in\X}\mathbf E(x)'\) such that
\(\bar\omega(x)\in\mathbf E(x)'\) for every \(x\in\X\) and
\[
\X\ni x\longmapsto\langle\bar\omega(x),\bar v(x)\rangle\in\R\;\;\;\text{is measurable},\quad\text{ for every }\bar v\in\bar\Gamma_0(\mathbf E).
\]
We introduce an equivalence relation on the space \(\bar\Gamma_0(\mathbf E'_{w^*})\): given any \(\bar\omega,\bar\eta\in\bar\Gamma_0(\mathbf E'_{w^*})\),
we declare that \(\bar\omega\sim\bar\eta\) if for any \(\bar v\in\bar\Gamma_0(\mathbf E)\) it holds that \(\langle\bar\omega(x)-\bar\eta(x),\bar v(x)\rangle=0\)
for \(\mm\)-a.e.\ \(x\in\X\).
\begin{remark}{\rm
In the case where the ambient Banach space \(\B\) is separable, it holds that
\[
\bar\omega\sim\bar\eta\quad\Longleftrightarrow\quad\bar\omega(x)=\bar\eta(x),\;\;\text{for }\mm\text{-a.e.\ }x\in\X.
\]
On the other hand, on arbitrary Banach spaces this needs not necessarily be the case.
\fr}\end{remark}
We denote the associated quotient space by
\[
\Gamma_0(\mathbf E'_{w^*})\coloneqq\bar\Gamma_0(\mathbf E'_{w^*})/\sim,
\]
while \(\pi_\mm\colon\bar\Gamma_0(\mathbf E'_{w^*})\to\Gamma_0(\mathbf E'_{w^*})\) stands for the projection map. Then the space
\(\Gamma_0(\mathbf E'_{w^*})\) is an \(L^0(\mm)\)-normed \(L^0(\mm)\)-module if endowed with the following \(L^0(\mm)\)-pointwise norm operator:
\[
|\omega|\coloneqq\bigvee\big\{\langle\bar\omega(\cdot),\bar v(\cdot)\rangle\;\big|\;\bar v\in\bar\Gamma_0(\mathbf E),\,|\bar v|\leq 1\big\},
\quad\text{ for every }\omega=\pi_\mm(\bar\omega)\in\Gamma_0(\mathbf E'_{w^*})
\]
\begin{remark}\label{rmk:fact_sep_B}{\rm
When \(\B\) is separable, it holds \(|\omega|(x)=\|\bar\omega(x)\|_{\mathbf E(x)'}\) for \(\mm\)-a.e.\ \(x\in\X\).
\fr}\end{remark}
In particular, for any given exponent \(q\in(1,\infty)\) we can consider the space
\[
\Gamma_q(\mathbf E'_{w^*})\coloneqq\big\{\omega\in\Gamma_0(\mathbf E'_{w^*})\;\big|\;|\omega|\in L^q(\mm)\big\},
\]
which inherits an \(L^q(\mm)\)-normed \(L^\infty(\mm)\)-module structure.
\medskip

Our interest towards the space \(\Gamma_q(\mathbf E'_{w^*})\) is motivated by the following result, which states that
the module dual of the section space \(\Gamma_p(\mathbf E)\) can be identified with \(\Gamma_q(\mathbf E'_{w^*})\) itself.
\begin{theorem}[Dual of a section space]\label{thm:char_dual}
Let \((\X,\Sigma,\mm)\) be a \(\sigma\)-finite measure space. Fix any exponent \(p\in(1,\infty)\).
Let \(\B\) be a Banach space and \(\mathbf E\) a Banach \(\B\)-bundle over \(\X\). Then
\[
\Gamma_q(\mathbf E'_{w^*})\cong\Gamma_p(\mathbf E)^*.
\]
An isomorphism \({\rm I}\colon\Gamma_q(\mathbf E'_{w^*})\to\Gamma_p(\mathbf E)^*\) of \(L^q(\mm)\)-normed \(L^\infty(\mm)\)-modules is given by the map
\begin{equation}\label{eq:def_I}
\langle{\rm I}(\omega),v\rangle\coloneqq\pi_\mm\big(\langle\bar\omega(\cdot),\bar v(\cdot)\rangle\big),\quad
\text{ for all }\omega=\pi_\mm(\bar\omega)\in\Gamma_q(\mathbf E'_{w^*})\text{ and }v=\pi_\mm(\bar v)\in\Gamma_p(\mathbf E).
\end{equation}
In particular, the space \(\Gamma_q(\mathbf E'_{w^*})\) is an \(L^q(\mm)\)-Banach \(L^\infty(\mm)\)-module.
\end{theorem}
\begin{proof}
Without loss of generality, we can suppose that \((\X,\Sigma,\mm)\) is a complete measure space. Let \(\mathcal I\) be the
differentiation basis on \((\X,\Sigma,\mm)\) given by Theorem \ref{thm:Leb_diff}. The validity of the  \(\mm\)-a.e.\ inequality
\(\big|\pi_\mm\big(\langle\bar\omega(\cdot),\bar v(\cdot)\rangle\big)\big|\leq|\omega||v|\) implies that \(\rm I\) is a well-defined homomorphism
of \(L^q(\mm)\)-normed \(L^\infty(\mm)\)-modules satisfying \(|{\rm I}(\omega)|\leq|\omega|\) \(\mm\)-a.e.\ for every \(\omega\in\Gamma_q(\mathbf E'_{w^*})\).
In order to conclude, it only remains to prove that the map \(\rm I\) is surjective and satisfies \(|{\rm I}(\omega)|\geq|\omega|\) \(\mm\)-a.e.\ for every
\(\omega\in\Gamma_q(\mathbf E'_{w^*})\). To this aim, let \(T\in\Gamma_p(\mathbf E)^*\) be fixed. Applying Theorem \ref{thm:approx_cont} to
\(|T|\in L^q(\mm)\), we obtain a representative \(\overline{|T|}\in\mathcal L^q(\Sigma)\) of \(|T|\) and an \(\mm\)-null set \(N\in\Sigma\)
such that \(\overline{|T|}\) is approximately continuous at each point of \(\X\setminus N\).
For any \(x\in\X\setminus N\), we define the linear space \(\mathcal V_x\subseteq\mathbf E(x)\) as
\(\mathcal V_x\coloneqq\{\hat v(x)\,:\,v\in\tilde{\mathcal V}_x\}\), where we set
\[
\tilde{\mathcal V}_x\coloneqq\Big\{v\in\Gamma_p(\mathbf E)\;\Big|\;x\text{ is of approximate continuity for some }\bar v\in\pi_\mm^{-1}(v)
\text{ and }x\in{\rm Leb}\big(T(v)\big)\Big\}.
\]
Moreover, we define the function \(\varphi_x\colon\mathcal V_x\to\R\) as
\begin{equation}\label{eq:def_varphi_x}
\varphi_x\big(\hat v(x)\big)\coloneqq\widehat{T(v)}(x),\quad\text{ for every }v\in\tilde{\mathcal V}_x.
\end{equation}
Let us check that \(\varphi_x\) is well-defined: if \(v,w\in\tilde{\mathcal V}_x\) satisfy \(\hat v(x)=\hat w(x)\), then Theorem \ref{thm:Leb_diff} yields
\[\begin{split}
\big|\widehat{T(v)}(x)-\widehat{T(w)}(x)\big|&=\bigg|\lim_{I\Rightarrow x}\fint_I\big(T(v)-T(w)\big)\,\d\mm\bigg|
\leq\limi_{I\Rightarrow x}\fint_I\big|T(v)-T(w)\big|\,\d\mm\\
&\leq\limi_{I\Rightarrow x}\fint_I|T||v-w|\,\d\mm\overset\star\leq\overline{|T|}(x)\limi_{I\Rightarrow x}\fint_I|v-w|\,\d\mm\\
&\leq\overline{|T|}(x)\lim_{I\Rightarrow x}\fint_I\big\|v(\cdot)-\hat v(x)\big\|_\B\,\d\mm
+\overline{|T|}(x)\lim_{I\Rightarrow x}\fint_I\big\|w(\cdot)-\hat w(x)\big\|_\B\,\d\mm\\
&=0,
\end{split}\]
whence it follows that the definition in \eqref{eq:def_varphi_x} is well-posed. Let us justify the starred inequality:
fix any \(\sigma\in\R\) for which there exists \(I\in\mathcal I_x\) such that \(\fint_J|T||v-w|\,\d\mm\geq\sigma\) for every \(J\in\mathcal I_x\)
with \(J\subseteq I\). By Theorem \ref{thm:approx_cont}, we can also assume that \(|v-w|\leq M\) holds \(\mm\)-a.e.\ on \(I\) for some constant \(M>0\).
Indeed, chosen two representatives \(\bar v\in\pi_\mm^{-1}(v)\) and \(\bar w\in\pi_\mm^{-1}(w)\) so that \(x\) is a point of approximate continuity both
for \(\bar v\) and for \(\bar w\), we have that there exist sets \(I_v,I_w\in\mathcal I_x\) satisfying \(\big\|\bar v(y)-\bar v(x)\big\|_\B\leq 1\)
for every \(y\in I_v\) and \(\big\|\bar w(y)-\bar w(x)\big\|_\B\leq 1\) for every \(y\in I_w\). Hence, picking any \(I\in\mathcal I_x\) with
\(I\subseteq I_v\cap I_w\), we have that\[
\big\|\bar v(y)-\bar w(y)\big\|_\B\leq\big\|\hat v(x)-\hat w(x)\big\|_\B+2\eqqcolon M,\quad\text{ for every }y\in I,
\]
which yields \(|v-w|\leq M\) \(\mm\)-a.e.\ on \(I\).
Given any \(\varepsilon>0\), we can find \(I_\varepsilon\in\mathcal I_x\) with \(I_\varepsilon\subseteq I\) such that the inequality
\(\overline{|T|}(y)\leq\overline{|T|}(x)+\varepsilon\) holds for all \(y\in I_\varepsilon\). In particular, for any \(J\in\mathcal I_x\)
such that \(J\subseteq I_\varepsilon\) we can estimate \(\overline{|T|}(x)\fint_J|v-w|\,\d\mm\geq\sigma-M\varepsilon\). Given that \(\varepsilon>0\)
was arbitrary, we deduce that \(\overline{|T|}(x)\limi_{I\Rightarrow x}\fint_I|v-w|\,\d\mm\geq\sigma\), whence the starred inequality above follows.

The linearity of \(\varphi_x\) can be easily proved. Moreover, for any \(v\in\tilde{\mathcal V}_x\) we can estimate
\[
\big|\varphi_x\big(\hat v(x)\big)\big|=\big|\widehat{T(v)}(x)\big|\leq\limi_{I\Rightarrow x}\fint_I|T||v|\,\d\mm
\overset\star\leq\overline{|T|}(x)\lim_{I\Rightarrow x}\fint_I|v|\,\d\mm=\overline{|T|}(x)\|\hat v(x)\|_{\mathbf E(x)},
\]
where the starred inequality can be justified exactly as before. This grants that the map \(\varphi_x\) is continuous and that is operator norm
does not exceed \(\overline{|T|}(x)\). Thanks to Hahn--Banach Theorem, we can find \(\bar\omega(x)\in\mathbf E(x)'\) such that
\(\bar\omega(x)|_{\mathcal V_x}=\varphi_x\) and \(\|\bar\omega(x)\|_{\mathbf E(x)'}\leq\overline{|T|}(x)\). Finally, for any point \(x\in N\)
we define \(\bar\omega(x)\coloneqq 0_{\mathbf E(x)'}\). We claim that the resulting map \(\bar\omega\) belongs to \(\bar\Gamma_0(\mathbf E'_{w^*})\).
In order to verify it, fix any \(\bar v\in\bar\Gamma_0(\mathbf E)\). Pick some partition \((A_n)_{n\in\N}\subseteq\Sigma\) of \(\X\) satisfying
\(v_n\coloneqq\1_{A_n}\cdot\bar v\in\bar\Gamma_p(\mathbf E)\) for every \(n\in\N\). Observe that \(\bar v_n(x)\in\mathcal V_x\)
for \(\mm\)-a.e.\ \(x\in\X\) by Theorem \ref{thm:Leb_diff}. Hence, calling \(v_n\coloneqq\pi_\mm(\bar v_n)\in\Gamma_p(\mathbf E)\)
for brevity, for \(\mm\)-a.e.\ \(x\in\X\) it holds that
\[
\langle\bar\omega(x),\bar v(x)\rangle=\sum_{n\in\N}\1_{A_n}(x)\,\langle\bar\omega(x),\bar v_n(x)\rangle=
\sum_{n\in\N}\1_{A_n}(x)\,\varphi_x\big(\hat v_n(x)\big)=\sum_{n\in\N}\1_{A_n}(x)\,\widehat{T(v_n)}(x).
\]
This implies that \(\X\ni x\mapsto\langle\bar\omega(x),\bar v(x)\rangle\in\R\) is measurable, thus accordingly \(\bar\omega\in\bar\Gamma_0(\mathbf E'_{w^*})\).

In order to conclude, it only remains to show that \({\rm I}(\omega)=T\) and \(|\omega|\leq|T|\) in the \(\mm\)-a.e.\ sense, where
\(\omega\in\Gamma_0(\mathbf E'_{w^*})\) stands for the equivalence class of \(\bar\omega\). Fix any \(v=\pi_\mm(\bar v)\in\Gamma_p(\mathbf E)\).
Then \(\hat v(x)\in\mathcal V_x\) for \(\mm\)-a.e.\ \(x\in\X\) by Theorem \ref{thm:Leb_diff}, so accordingly we have that
\[
\langle{\rm I}(\omega),v\rangle(x)=\langle\bar\omega(x),\bar v(x)\rangle=\varphi_x\big(\hat v(x)\big)=\widehat{T(v)}(x)=T(v)(x),
\quad\text{ for }\mm\text{-a.e.\ }x\in\X.
\]
This shows that \(\langle{\rm I}(\omega),v\rangle=T(v)\) for all \(v\in\Gamma_p(\mathbf E)\) and thus \({\rm I}(\omega)=T\).
Finally, recalling that \(\|\bar\omega(x)\|_{\mathbf E(x)'}\leq\overline{|T|}(x)\) for \(\mm\)-a.e.\ \(x\in\X\), we deduce
that for any \(\bar v\in\bar\Gamma_0(\mathbf E)\) with \(|\bar v|\leq 1\) it holds
\[
\langle\bar\omega(x),\bar v(x)\rangle\leq\|\bar\omega(x)\|_{\mathbf E(x)'}\|\bar v(x)\|_{\mathbf E(x)}\leq\overline{|T|}(x),
\quad\text{ for }\mm\text{-a.e.\ }x\in\X,
\]
so that \(|\omega|\leq|T|\) holds in the \(\mm\)-a.e.\ sense. Therefore, the statement is achieved.
\end{proof}
\subsection{Reflexive bundles/modules}\label{s:reflex}
In this section, we will prove that the section space of a separable Banach bundle is reflexive if and only if (almost all)
its fibers are reflexive. Before stating the main theorem, we need to discuss a few auxiliary results.
\begin{remark}\label{rmk:sep_predual}{\rm
Let us recall a standard fact in Banach space theory. Let \(\B\) be a Banach space whose dual \(\B'\) is separable. Let \((v_n)_{n\in\N}\subseteq\B\) and
\((\omega_n)_{n\in\N}\subseteq\B'\) be given sequences satisfying \(\|v_n\|_\B=\|\omega_n\|_{\B'}=\langle\omega_n,v_n\rangle=1\) for every \(n\in\N\).
Suppose \((\omega_n)_{n\in\N}\) is dense in the unit sphere \(\mathbb S_{\B'}\). Then the \(\mathbb Q\)-linear subspace of \(\B\) generated by
\((v_n)_{n\in\N}\) is dense in \(\B\).
\fr}\end{remark}
\begin{proposition}\label{prop:partial_reflex_impl}
Let \((\X,\Sigma,\mm)\) be a \(\sigma\)-finite measure space and \(p\in(1,\infty)\) a given exponent. Let \(\B\) be a separable Banach space and let
\(\mathbf E\) be a reflexive Banach \(\B\)-bundle over \(\X\). Consider the mapping \(\theta\colon\Gamma_p(\mathbf E)\to\Gamma_q(\mathbf E'_{w^*})^*\),
which is given by
\begin{equation}\label{eq:def_theta}
\langle\theta(v),\omega\rangle\coloneqq\pi_\mm\big(\langle\bar\omega(\cdot),\bar v(\cdot)\rangle\big),
\quad\text{ for all }v=\pi_\mm(\bar v)\in\Gamma_p(\mathbf E)\text{ and }\omega=\pi_\mm(\bar\omega)\in\Gamma_q(\mathbf E'_{w^*}).
\end{equation}
Then the operator \(\theta\) is an isomorphism of \(L^p(\mm)\)-Banach \(L^\infty(\mm)\)-modules.
\end{proposition}
\begin{proof}
Without loss of generality, we can suppose that \((\X,\Sigma,\mm)\) is a complete measure space.\\
{\color{blue}\textsc{Step 1.}} The \(\mm\)-a.e.\ inequality \(\big|\pi_\mm(\langle\bar\omega(\cdot),\bar v(\cdot)\rangle)\big|\leq|v||\omega|\)
ensures that \(\theta\) is a well-defined homomorphism of \(L^p(\mm)\)-Banach \(L^\infty(\mm)\)-modules satisfying \(|\theta(v)|\leq|v|\)
for every \(v\in\Gamma_p(\mathbf E)\).\\
{\color{blue}\textsc{Step 2.}} It remains to prove that \(\theta\) is surjective and satisfies \(|\theta(v)|\geq|v|\) for every
\(v\in\Gamma_p(\mathbf E)\). To this aim, fix any \(L\in\Gamma_q(\mathbf E'_{w^*})^*\). Pick a sequence \((v_n)_{n\in\N}\subseteq\Gamma_p(\mathbf E)\)
with \(|v_n|(x)\in\{0,1\}\) for every \(n\in\N\) and \(\mm\)-a.e.\ \(x\in\X\), and such that
\[
\{v_n(x)\,:\,n\in\N\}\setminus\{0_{\mathbf E(x)}\}\;\;\text{is dense in }\mathbb S_{\mathbf E(x)},\quad\text{ for }\mm\text{-a.e.\ }x\in\X.
\]
Given any \(n\in\N\), thanks to \cite[Corollary 1.2.16]{Gigli14} we can find an element \(\tilde\omega_n\in\Gamma_p(\mathbf E)^*\)
such that \(|\tilde\omega_n|=|v_n|=\langle\tilde\omega_n,v_n\rangle\) holds \(\mm\)-a.e.\ on \(\X\).
Now define \(w_n\coloneqq{\rm I}^{-1}(\tilde w_n)\in\Gamma_q(\mathbf E'_{w^*})\), where
\({\rm I}\colon\Gamma_q(\mathbf E'_{w^*})\to\Gamma_p(\mathbf E)^*\) stands for the isomorphism provided by Theorem \ref{thm:char_dual}.
Let us denote by \(\mathcal V\) the \(\mathbb Q\)-linear subspace of \(\Gamma_q(\mathbf E'_{w^*})\) generated by \((w_n)_{n\in\N}\).
Notice that \(\mathcal V\) is a countable family by construction. Given \(n\in\N\) and \(w\in\mathcal V\), fix representatives
\(\bar v_n\), \(\bar\omega\), \(\overline{L(\omega)}\), and \(\overline{|L|}\) of \(v_n\), \(\omega\), \(L(\omega)\), and \(|L|\), respectively.
By Remark \ref{rmk:fact_sep_B}, there exists \(N\in\Sigma\) with \(\mm(N)=0\) such that for any \(x\in\X\setminus N\) it holds that
\begin{subequations}\begin{align}
\label{eq:choice_N_-1}\mathbf E(x)&,\quad\text{ is reflexive,}\\
\label{eq:choice_N_0}\{\bar v_m(x)\,:\,m\in\N\}\setminus\{0_{\mathbf E(x)}\}&,\quad\text{ is dense in }\mathbb S_{\mathbf E(x)},\\
\label{eq:choice_N_1}\|\bar\omega_n(x)\|_{\mathbf E(x)'}&=\|\bar v_n(x)\|_{\mathbf E(x)}=\langle\bar\omega_n(x),\bar v_n(x)\rangle,\\
\label{eq:choice_N_2}\overline{(\omega+\eta)}(x)&=\bar\omega(x)+\bar\eta(x),\\
\label{eq:choice_N_3}\overline{(\lambda\,\omega)}(x)&=\lambda\,\bar\omega(x),\\
\label{eq:choice_N_4}\overline{L(\omega+\eta)}(x)&=\overline{L(\omega)}(x)+\overline{L(\eta)}(x),\\
\label{eq:choice_N_5}\overline{L(\lambda\,\omega)}(x)&=\lambda\,\overline{L(\omega)}(x),\\
\label{eq:choice_N_6}\overline{L(\omega)}(x)&\leq\overline{|L|}(x)\|\bar\omega(x)\|_{\mathbf E(x)'},
\end{align}\end{subequations}
for every \(n\in\N\), \(\omega,\eta\in\mathcal V\), and \(\lambda\in\mathbb Q\). Given any \(x\in\X\setminus N\), let us consider the countable,
\(\mathbb Q\)-linear subspace \(\mathcal V_x\coloneqq\{\bar\omega(x)\,:\,\omega\in\mathcal V\}\) of \(\mathbf E(x)'\). The fact that \(\mathcal V_x\)
is a \(\mathbb Q\)-linear space is granted by \eqref{eq:choice_N_2} and \eqref{eq:choice_N_3}. By taking \eqref{eq:choice_N_-1}, \eqref{eq:choice_N_0},
\eqref{eq:choice_N_1}, and Remark \ref{rmk:sep_predual} into account, we deduce that \(\mathcal V_x\) is dense in \(\mathbf E(x)'\). Now we define
the function \(\varphi_x\colon\mathcal V_x\to\R\) as
\[
\varphi_x\big(\bar\omega(x)\big)\coloneqq\overline{L(\omega)}(x),\quad\text{ for every }\omega\in\mathcal V.
\]
The well-posedness of \(\varphi_x\) stems from the observation that for any \(\omega,\eta\in\mathcal V\) it holds that
\[\begin{split}
\big|\overline{L(\omega)}(x)-\overline{L(\eta)}(x)\big|&\overset{\eqref{eq:choice_N_4}}=\big|\overline{L(\omega-\eta)}(x)\big|
\overset{\eqref{eq:choice_N_6}}\leq\overline{|L|}(x)\big\|\overline{(\omega-\eta)}(x)\big\|_{\mathbf E(x)'}\\
&\overset{\eqref{eq:choice_N_2}}=\overline{|L|}(x)\big\|\bar\omega(x)-\bar\eta(x)\big\|_{\mathbf E(x)'}.
\end{split}\]
The \(\mathbb Q\)-linearity of \(\varphi_x\) is a consequence of \eqref{eq:choice_N_4} and \eqref{eq:choice_N_5}. Moreover, \eqref{eq:choice_N_6} grants
the validity of the inequality \(\big|\varphi_x\big(\bar\omega(x)\big)\big|\leq\overline{|L|}(x)\|\bar\omega(x)\|_{\mathbf E(x)'}\) for every
\(\omega\in\mathcal V\), whence the continuity of the function \(\varphi_x\) follows. Therefore, there exists a unique element
\(\bar v(x)\in\mathbf E(x)\cong\mathbf E(x)''\) such that \(\langle\bar\omega(x),\bar v(x)\rangle=\overline{L(\omega)}(x)\)
holds for every \(\omega\in\mathcal V\) and \(\|\bar v(x)\|_{\mathbf E(x)} \leq\overline{|L|}(x)\). Finally, for any point \(x\in N\)
we define \(\bar v(x)\coloneqq 0_{\mathbf E(x)}\).\\
{\color{blue}\textsc{Step 3.}} Next we claim that the resulting map \(\bar v\) belongs to \(\bar\Gamma_0(\mathbf E)\). By virtue of the separability
of \(\B\), it is sufficient to prove that \(\bar v\colon\X\to\B\) is weakly measurable. To this aim, fix any element \(\eta_0\in\B'\). Define
\(\bar\eta(x)\coloneqq\eta_0|_{\mathbf E(x)}\in\mathbf E(x)'\) for every \(x\in\X\). For any \(\omega\in\mathcal V\), one has that
\begin{equation}\label{eq:expr_bareta}
\big\|\bar\eta(x)-\bar\omega(x)\big\|_{\mathbf E(x)'}=\sup_{n\in\N}\big\langle\eta_0-\bar\omega(x),\bar v_n(x)\big\rangle,
\quad\text{ for }\mm\text{-a.e.\ }x\in\X.
\end{equation}
Since the function \(\X\ni x\mapsto\big\|\bar\eta(x)-\bar\omega(x)\big\|_{\mathbf E(x)'}\) is measurable for every \(\omega\in\mathcal V\)
thanks to \eqref{eq:expr_bareta} and the space \(\mathcal V_x\) is dense in \(\mathbf E(x)'\) for \(\mm\)-a.e.\ \(x\in\X\), we deduce that
for any \(k\in\N\) we can find a partition \(\{A^k_\omega\}_{\omega\in\mathcal V}\subseteq\Sigma\) of \(\X\) (up to \(\mm\)-null sets) such
that \(\big\|\bar\eta(x)-\bar\eta_k(x)\big\|_{\mathbf E(x)'}\leq 1/k\) for \(\mm\)-a.e.\ \(x\in\X\), where we set
\(\bar\eta_k\coloneqq\sum_{\omega\in\mathcal V}\1_{A^k_\omega}\,\bar\omega\). Therefore, for \(\mm\)-a.e.\ \(x\in\X\) we can express
\[\begin{split}
\langle\eta_0,\bar v(x)\rangle&=\langle\bar\eta(x),\bar v(x)\rangle=\lim_{k\to\infty}\langle\bar\eta_k(x),\bar v(x)\rangle=
\lim_{k\to\infty}\sum_{\omega\in\mathcal V}\1_{A^k_\omega}(x)\,\langle\bar\omega(x),\bar v(x)\rangle\\
&=\lim_{k\to\infty}\sum_{\omega\in\mathcal V}\1_{A^k_\omega}(x)\,\overline{L(\omega)}(x),
\end{split}\]
thus accordingly \(\langle\eta_0,\bar v(\cdot)\rangle\) is measurable. By arbitrariness of \(\eta_0\in\B'\), we conclude that
\(\bar v\) is weakly (thus, strongly) measurable, as desired. Let us then define \(v\coloneqq\pi_\mm(\bar v)\in\Gamma_0(\mathbf E)\).\\
{\color{blue}\textsc{Step 4.}} In order to conclude, it only remains to show that \(\theta(v)=L\) and \(|v|\leq|L|\) in the \(\mm\)-a.e.\ sense.
Fix any \(\eta\in\Gamma_q(\mathbf E'_{w^*})\), with representative \(\bar\eta\in\bar\Gamma_0(\mathbf E'_{w^*})\). By arguing as we did
in \textsc{Step 3}, we can construct a sequence \((\bar\eta_k)_{k\in\N}\subseteq\bar\Gamma_0(\mathbf E'_{w^*})\) of the form
\(\bar\eta_k=\sum_{\omega\in\mathcal V}\1_{A^k_\omega}\,\bar\omega\), such that \(\lim_k\big\|\bar\eta_k(x)-\bar\eta(x)\big\|_{\mathbf E(x)'}=0\)
for \(\mm\)-a.e.\ \(x\in\X\).
Therefore, for \(\mm\)-a.e.\ \(x\in\X\) it holds that
\[\begin{split}
\langle\theta(v),\eta\rangle(x)&=\langle\bar\eta(x),\bar v(x)\rangle=\lim_{k\to\infty}\langle\bar\eta_k(x),\bar v(x)\rangle
=\lim_{k\to\infty}\sum_{\omega\in\mathcal V}\1_{A^k_\omega}(x)\,\langle\bar\omega(x),\bar v(x)\rangle\\
&=\lim_{k\to\infty}\sum_{\omega\in\mathcal V}\1_{A^k_\omega}(x)\,\overline{L(\omega)}(x)=\lim_{k\to\infty}L\big(\pi_\mm(\bar\eta_k)\big)(x)=L(\eta)(x).
\end{split}\]
By arbitrariness of \(\eta\in\Gamma_q(\mathbf E'_{w^*})\), it follows that \(\theta(v)=L\). Finally, since \(\|\bar v(x)\|_{\mathbf E(x)}\leq\overline{|L|}(x)\)
for \(\mm\)-a.e.\ \(x\in\X\), we obtain the \(\mm\)-a.e.\ inequality \(|v|\leq|L|\). Hence, the statement is achieved.
\end{proof}
Proposition \ref{prop:partial_reflex_impl} will play a key role in proving one implication of the main result of this section, namely,
Theorem \ref{thm:reflex_impl}. In order to prove the converse implication, we need the alternative -- more `quantitative' -- characterisation
of reflexivity that we report in Lemma \ref{lem:equiv_not_reflex}. Before passing to its statement, it is convenient to introduce
some additional notation.
\medskip

We denote by \(\bigoplus_\N\mathbb Q\) the set of sequences \(\boldsymbol q=(q_i)_{i\in\N}\in\mathbb Q^\N\)
satisfying \(q_i=0\) for all but finitely many \(i\in\N\). We define
\(\Delta\coloneqq\big\{\boldsymbol q\in\bigoplus_\N\mathbb Q\cap[0,1]^\N\,:\,\sum_{i\in\N}q_i=1\big\}\).
Given any \(\boldsymbol q,\boldsymbol r\in\Delta\), we declare that \(\boldsymbol q\prec\boldsymbol r\)
if \(\max\{i\in\N\,:\,q_i\neq 0\}<\min\{j\in\N\,:\,r_j\neq 0\}\). Finally, we define
\[
\mathscr F\coloneqq\big\{(\boldsymbol q,\boldsymbol r)\in\Delta\times\Delta\;\big|\;\boldsymbol q\prec\boldsymbol r\big\}.
\]
\begin{lemma}\label{lem:equiv_not_reflex}
Let \(\B\) be a Banach space. Then the following two conditions are equivalent:
\begin{itemize}
\item[\(\rm i)\)] \(\B\) is not reflexive.
\item[\(\rm ii)\)] Given any \(\lambda\in(0,1)\), there exists a sequence \((v_i)_{i\in\N}\subset B_\B\) such that
\begin{equation}\label{eq:equiv_refl_Banach}
\bigg\|\sum_{i\in\N}q_i v_i-\sum_{i\in\N}r_i v_i\bigg\|_\B\geq\lambda,\quad\text{ for every }(\boldsymbol q,\boldsymbol r)\in\mathscr F.
\end{equation}
\end{itemize}
\end{lemma}
\begin{proof}
The Eberlein--\v{S}mulian Theorem (see, \emph{e.g.}, \cite[Theorems 3.18 and 3.19]{Brezis11}) says that
\(\B\) is reflexive if and only if every sequence in \(B_\B\) admits a weakly converging subsequence. Then:\\
{\color{blue}\({\rm i)}\Rightarrow{\rm ii)}.\)} It readily follows, \emph{e.g.}, from \cite[Theorem 3.132]{FHHMZ11}.\\
{\color{blue}\({\rm ii)}\Rightarrow{\rm i)}.\)} Let \((v_i)_{i\in\N}\subset B_\B\) satisfy \eqref{eq:equiv_refl_Banach}. We argue by contradiction:
suppose \(\B\) is reflexive. Then Mazur's Lemma (see, \emph{e.g.}, \cite[Corollary 3.8]{Brezis11}) yields an element
\(v\in B_\B\) and a sequence \((\boldsymbol q^j)_{j\in\N}\subset\Delta\) such that \(\boldsymbol q^j\prec\boldsymbol q^{j+1}\)
for every \(j\in\N\) and \(\sum_{i\in\N}q^j_i v_i\to v\) strongly in \(\B\) as \(j\to\infty\). In particular,
\(\big\|\sum_{i\in\N}q^j_i v_i-\sum_{i\in\N}q^{j+1}_i v_i\big\|_\B<\lambda\) for \(j\in\N\) big enough, contradicting ii).
\end{proof}
Combining Proposition \ref{prop:partial_reflex_impl} with Lemma \ref{lem:equiv_not_reflex}, we obtain the main result of this section:
\begin{theorem}[Reflexive bundles/modules]\label{thm:reflex_impl}
Let \((\X,\Sigma,\mm)\) be a \(\sigma\)-finite measure space, \(\B\) a separable Banach space, and \(\mathbf E\)
a Banach \(\B\)-bundle over \(\X\). Then \(\mathbf{E}\) is a reflexive bundle if and only if \(\Gamma_p(\mathbf E)\)
is a reflexive Banach space for every (or, equivalently, for some) \(p\in(1,\infty)\).
\end{theorem}
\begin{proof}
\ \\
{\color{blue}\textsc{Necessity.}}
Suppose that \(\mathbf E\) is a reflexive bundle and fix any exponent \(p\in(1,\infty)\).
We call \({\rm I}\colon\Gamma_q(\mathbf E'_{w^*})\to\Gamma_p(\mathbf E)^*\) the isomorphism provided by Theorem \ref{thm:char_dual}.
Call \({\rm J}\colon\Gamma_p(\mathbf E)^*\to\Gamma_q(\mathbf E'_{w^*})\) its inverse and consider the adjoint
\({\rm J}^{\rm ad}\colon\Gamma_q(\mathbf E'_{w^*})^*\to\Gamma_p(\mathbf E)^{**}\) of the isomorphism \(\rm J\). Moreover, let
\(\theta\colon\Gamma_p(\mathbf E)\to\Gamma_q(\mathbf E'_{w^*})^*\) be the isomorphism given by Proposition \ref{prop:partial_reflex_impl}.
By unwrapping the various definitions, it can be readily checked that
\[\begin{tikzcd}
\Gamma_p(\mathbf E) \arrow[r,"\theta"] \arrow[rd,swap,"J_{\Gamma_p(\mathbf E)}"]
& \Gamma_q(\mathbf E'_{w^*})^* \arrow[d,"{\rm J}^{\rm ad}"] \\ &\Gamma_p(\mathbf E)^{**}
\end{tikzcd}\]
is a commutative diagram. Indeed, let us fix any \(v=\pi_\mm(\bar v)\in\Gamma_p(\mathbf E)\) and \(T\in\Gamma_p(\mathbf E)^*\).
Also, define \(\omega\coloneqq{\rm J}(T)\in\Gamma_q(\mathbf E'_{w^*})\) and pick a representative \(\bar\omega\in\bar\Gamma_0(\mathbf E'_{w^*})\)
of \(\omega\). Then we have that
\[\begin{split}
\langle({\rm J}^{\rm ad}\circ\theta)(v),T\rangle&\overset{\eqref{eq:def_ad}}=\langle\theta(v),{\rm J}(T)\rangle
=\langle\theta(v),\omega\rangle\overset{\eqref{eq:def_theta}}=\pi_\mm\big(\langle\bar\omega(\cdot),\bar v(\cdot)\rangle\big)
\overset{\eqref{eq:def_I}}=\langle{\rm I}(\omega),v\rangle\\
&\overset{\phantom{\eqref{eq:def_ad}}}=\langle T,v\rangle\overset{\eqref{eq:def_James}}=\langle J_{\Gamma_p(\mathbf E)}(v),T\rangle,
\end{split}\]
yielding \({\rm J}^{\rm ad}\circ\theta=J_{\Gamma_p(\mathbf E)}\). Therefore, \(J_{\Gamma_p(\mathbf E)}\) is an isomorphism and thus
\(\Gamma_p(\mathbf E)\) is reflexive.\\
{\color{blue}\textsc{Sufficiency.}}
Suppose that \(\Gamma_p(\mathbf{E})\) is reflexive for some \(p\in(1,\infty)\). Using Proposition \ref{prop:spanning_sects}, we obtain
a countable family \(Z\subseteq\bar\Gamma_\infty({\bf E})\) such that \(\|v(x)\|_\B\leq 1\) for every \((v,x)\in Z\times\X\) and
\[
\big\{v(x)\;\big|\;v\in Z\big\}\text{ is dense in }B_{{\bf E}(x)},\quad\text{ for every }x\in\X.
\]
We equip \(Z\) with the discrete topology and \(Z^\N\) with the product topology. Then \(Z^\N\) is a Polish space (\emph{i.e.},
a metrisable space whose topology is induced by a complete, separable distance), which is homeomorphic to the Baire space \(\N^\N\)
(see, for example, \cite[Section 3.14]{AliprantisBorder99}). We define \(\boldsymbol\varphi\colon\X\twoheadrightarrow Z^\N\) as
\(\boldsymbol\varphi(x)\coloneqq\big\{\boldsymbol v\in Z^\N\;\big|\;(\boldsymbol v,x)\in H\big\}\) for every \(x\in\X\), where we set
\[
H\coloneqq\bigcap_{(\boldsymbol q,\boldsymbol r)\in\mathscr F}\Bigg\{(\boldsymbol v,x)\in Z^\N\times\X
\;\Bigg|\;\bigg\|\sum_{i\in\N}q_i v_i(x)-\sum_{i\in\N}r_i v_i(x)\bigg\|_\B\geq\frac{1}{2}\Bigg\}.
\]
Recalling that a base for the topology of \(Z^\N\) is given by those sets of the form
\[
\{v_1\}\times\dots\times\{v_n\}\times Z\times Z\times\dots,\quad\text{ with }n\in\N\text{ and }v_1,\ldots,v_n\in Z,
\]
one can readily check that \(\boldsymbol\varphi\) is a weakly measurable map from \(\X\) to \(Z^\N\) having closed values.

We now argue by contradiction: suppose that there exists \(P\in\Sigma\) with \(0<\mm(P)<+\infty\) such that \({\bf E}(x)\)
is not reflexive for every \(x\in P\). Applying Lemma \ref{lem:equiv_not_reflex} to each \({\bf E}(x)\) with \(x\in P\),
we deduce that \(\boldsymbol\varphi(x)\neq\varnothing\) for every \(x\in P\). Thanks to the Kuratowski--Ryll-Nardzewski Selection Theorem
(see, \emph{e.g.}, \cite[Theorem 18.13]{AliprantisBorder99}), we can find a measurable maping \(V\colon P\to Z^\N\) such that
\(V(x)\in\boldsymbol\varphi(x)\) for every \(x\in P\). For any \(i\in\N\), we denote by \(\pi_i\colon Z^\N\to Z\) the projection onto the
\(i\)-th component, which is continuous by definition of the product topology. Then \(\pi_i\circ V\colon P\to Z\) is measurable, so that
\(P^i_v\coloneqq(\pi_i\circ V)^{-1}(\{v\})\in\Sigma\) for every \(v\in Z\) and \((P^i_v)_{v\in Z}\) is a partition of \(P\).
Given any \(i\in\N\), we define \(\bar v_i\colon\X\to\B\) as
\[
\bar v_i(x)\coloneqq\frac{(\pi_i\circ V)(x)(x)}{\mm(P)^{1/p}}\in{\bf E}(x),\quad\text{ for every }x\in P,
\]
and \(\bar v_i(x)\coloneqq 0_\B\) for all \(x\in\X\setminus P\). Note that \(\bar v_i(x)=\sum_{v\in Z}\mm(P)^{-1/p}\1_{P^i_v}(x)v(x)\)
for all \(x\in P\), thus in particular \(\bar v_i\in\bar\Gamma_\infty({\bf E})\cap\bar\Gamma_p({\bf E})\) and
\(\|\pi_\mm(\bar v_i)\|_{\Gamma_p({\bf E})}\leq 1\). Observe also that it holds
\begin{equation}\label{eq:impl_2_reflex_aux}
\bigg\|\sum_{i\in\N}q_i\bar v_i(x)-\sum_{i\in\N}r_i\bar v_i(x)\bigg\|_\B\geq\frac{1}{2\mm(P)^{1/p}},
\quad\text{ for every }(\boldsymbol q,\boldsymbol r)\in\mathscr F\text{ and }x\in P.
\end{equation}
Hence, denoting by \(v_i\in\Gamma_p({\bf E})\) the equivalence class of \(\bar v_i\),
for any \((\boldsymbol q,\boldsymbol r)\in\mathscr F\) we can estimate
\[\begin{split}
\bigg\|\sum_{i\in\N}q_i v_i-\sum_{i\in\N}r_i v_i\bigg\|_{\Gamma_p({\bf E})}=
\bigg(\int_P\bigg\|\sum_{i\in\N}q_i\bar v_i(x)-\sum_{i\in\N}r_i\bar v_i(x)\bigg\|_\B^p\,\d\mm(x)\bigg)^{1/p}
\overset{\eqref{eq:impl_2_reflex_aux}}\geq\frac{1}{2}.
\end{split}\]
Using Lemma \ref{lem:equiv_not_reflex} again, we deduce that \(\Gamma_p({\bf E})\) is not reflexive, leading to a contradiction.
\end{proof}
\begin{remark}\label{rmk:on_proof_refl}{\rm
As we already mentioned in the Introduction, the proof of the sufficiency part of Theorem \ref{thm:reflex_impl}
follows along the lines sketched in the proof of \cite[Theorem 6.19]{HLR91}. On the other hand, the proof of the
necessity part is different from the one of \cite[Theorem 6.19]{HLR91}, and in particular it avoids the use of
Rosenthal's \(\ell^1\)-Theorem.
\fr}\end{remark}
\appendix
\section{A criterion to detect Banach modules}\label{app:criterion_nmod}
Aim of this appendix is to address the following problem: \emph{given a module \(\mathscr M\) over \(L^\infty(\mm)\), can we characterise those complete norms on
\(\mathscr M\) that come from an \(L^p(\mm)\)-pointwise norm?} We will provide a positive answer to this question in Theorem \ref{thm:when_norm_ind_ptwse_norm} below.
\medskip

First, we recall a well-known, elementary result concerning Radon--Nikod\'{y}m derivatives. We report its proof for the reader's usefulness.
\begin{lemma}\label{lem:ineq_meas}
Let \((\X,\Sigma)\) be a measurable space. Let \(\mm,\mu_1,\mu_2,\mu_3\) be \(\sigma\)-finite measures on \(\Sigma\) such that \(\mu_1,\mu_2,\mu_3\ll\mm\).
Let \(\alpha\in(0,+\infty)\) be given. Suppose that
\begin{equation}\label{eq:ineq_meas_cl1}
\mu_1(E)^\alpha\leq\mu_2(E)^\alpha+\mu_3(E)^\alpha,\quad\text{ for every }E\in\Sigma.
\end{equation}
Then it holds that
\begin{equation}\label{eq:ineq_meas_cl2}
\bigg(\frac{\d\mu_1}{\d\mm}\bigg)^\alpha\leq\bigg(\frac{\d\mu_2}{\d\mm}\bigg)^\alpha+\bigg(\frac{\d\mu_3}{\d\mm}\bigg)^\alpha,\quad\text{ in the }\mm\text{-a.e.\ sense.}
\end{equation}
\end{lemma}
\begin{proof}
Let us denote \(f_j\coloneqq\frac{\d\mu_j}{\d\mm}\) for \(j=1,2,3\). Let \(k\in\N\) be fixed. By using the \(\sigma\)-finiteness of \(\mm\),
we can find a partition \((E_i)_{i\in\N}\subseteq\Sigma\) such that \(0<\mm(E_i)<+\infty\) for every \(i\in\N\) and
\begin{equation}\label{eq:ineq_meas_aux1}
\big|f_j(x)-f_j(y)\big|\leq\frac{1}{k},\quad\text{ for every }i\in\N,\,j=1,2,3,\text{ and }\mm\text{-a.e.\ }x,y\in E_i.
\end{equation}
Define \(\lambda_{ij}\coloneqq\fint_{E_i}f_j\,\d\mm\) for every \(i\in\N\) and \(j=1,2,3\). Observe that \eqref{eq:ineq_meas_aux1} ensures that
\begin{equation}\label{eq:ineq_meas_aux2}
\big|f_j(x)-\lambda_{ij}\big|\leq\frac{1}{k},\quad\text{ for every }i\in\N,\,j=1,2,3,\text{ and }\mm\text{-a.e.\ }x\in E_i.
\end{equation}
Given that \(\int_{E_i}f_j\,\d\mm=\int_{E_i}\frac{\d\mu_j}{\d\mm}\,\d\mm=\mu_j(E_i)\), we deduce that
\begin{equation}\label{eq:ineq_meas_aux3}
\lambda_{i1}^\alpha=\frac{\mu_1(E_i)^\alpha}{\mm(E_i)^\alpha}\overset{\eqref{eq:ineq_meas_cl1}}\leq
\frac{\mu_2(E_i)^\alpha}{\mm(E_i)^\alpha}+\frac{\mu_3(E_i)^\alpha}{\mm(E_i)^\alpha}=\lambda_{i2}^\alpha+\lambda_{i3}^\alpha,\quad\text{ for every }i\in\N.
\end{equation}
Hence, by combining \eqref{eq:ineq_meas_aux2} with \eqref{eq:ineq_meas_aux3} we see that for every \(i\in\N\) and \(\mm\)-a.e.\ \(x\in E_i\) it holds
\[
\bigg(f_1(x)-\frac{1}{k}\bigg)^\alpha\leq\lambda_{i1}^\alpha\leq\lambda_{i2}^\alpha+\lambda_{i3}^\alpha
\leq\bigg(f_2(x)+\frac{1}{k}\bigg)^\alpha+\bigg(f_3(x)+\frac{1}{k}\bigg)^\alpha.
\]
By arbitrariness of \(i,k\in\N\), we conclude that \(f_1^\alpha\leq f_2^\alpha+f_3^\alpha\) holds \(\mm\)-a.e., yielding \eqref{eq:ineq_meas_cl2}.
\end{proof}
We are in a position to characterise which complete norms \(\|\cdot\|\) on an \(L^\infty(\mm)\)-module \(\mathscr M\) are induced by an
\(L^p(\mm)\)-pointwise norm. Roughly speaking, the required compatibility between the norm and the module structure is expressed via
two conditions, labelled \(2\rm a)\) and \(2\rm b)\): the former relates the given norm with the multiplication by \(L^\infty(\mm)\)-functions
and the chosen exponent \(p\), while the latter is a weak continuity assumption on the multiplication operator.
\begin{theorem}[When a norm is induced by a pointwise norm]\label{thm:when_norm_ind_ptwse_norm}
Let \((\X,\Sigma,\mm)\) be a \(\sigma\)-finite measure space. Let \(\mathscr M\) be a module over the ring \(L^\infty(\mm)\) and \(\|\cdot\|\)
a complete norm on \(\mathscr M\). Let \(p\in[1,\infty)\) be a given exponent. Then the following two conditions are equivalent:
\begin{itemize}
\item[\(1)\)] There exists an \(L^p(\mm)\)-pointwise norm operator \(|\cdot|\colon\mathscr M\to L^p(\mm)\) on \(\mathscr M\) such that
\[
\|v\|=\big\||v|\big\|_{L^p(\mm)},\quad\text{ for every }v\in\mathscr M.
\]
\item[\(2)\)] The following two properties are satisfied:
\begin{itemize}
\item[\(2\rm a)\)] It holds that \(\|\1_E\cdot v\|^p+\|\1_{\X\setminus E}\cdot v\|^p=\|v\|^p\) for every \(E\in\Sigma\) and \(v\in\mathscr M\).
\item[\(2\rm b)\)] It holds \(\lim_{n\to\infty}\|f_n\cdot v\|=0\) for every \(v\in\mathscr M\) and for every \((f_n)_{n\in\N}\subseteq L^\infty(\mm)\)
such that \(f_n\rightharpoonup 0\) weakly\(^*\) in \(L^\infty(\mm)\) as \(n\to\infty\).
\end{itemize}
\end{itemize}
\end{theorem}
\begin{proof}
\ \\
{\color{blue}\(1)\Longrightarrow 2)\).} Suppose \(1)\) holds. Let us prove that \(2\rm a)\) is satisfied. Fix \(E\in\Sigma\) and \(v\in\mathscr M\). Then
\[
\|\1_E\cdot v\|^p+\|\1_{\X\setminus E}\cdot v\|^p=\int_E|v|^p\,\d\mm+\int_{\X\setminus E}|v|^p\,\d\mm=\int|v|^p\,\d\mm=\|v\|^p,
\]
thus \(2\rm a)\) holds. To prove \(2\rm b)\), fix any sequence \((f_n)_{n\in\N}\subseteq L^\infty(\mm)\) such that \(f_n\rightharpoonup 0\) weakly\(^*\)
in \(L^\infty(\mm)\). This yields \(M\coloneqq\sup_n\|f_n\|_{L^\infty(\mm)}<+\infty\). Therefore, since \(|v|^p\in L^1(\mm)\), we have
\[
\lims_{n\to\infty}\|f_n\cdot v\|=\lims_{n\to\infty}\bigg(\int|f_n\cdot v|^p\,\d\mm\bigg)^{1/p}\leq M^{(p-1)/p}\lim_{n\to\infty}\bigg(\int|f_n||v|^p\,\d\mm\bigg)^{1/p}=0,
\]
thus \(2\rm b)\) holds. All in all, \(2)\) is proven.\\
{\color{blue}\(2)\Longrightarrow 1)\).} Suppose \(2)\) holds. First of all, we claim that for any \(v\in\mathscr M\) one has that
\begin{equation}\label{eq:when_norm_ind_ptwse_norm_aux1}
\|\1_E\cdot v\|^p=\sum_{n\in\N}\|\1_{E_n}\cdot v\|^p,\quad\text{ if }(E_n)_{n\in\N}\subseteq\Sigma\text{ are pairwise disjoint and }E\coloneqq\bigcup_{n\in\N}E_n.
\end{equation}
In order to prove it, denote \(E'_n\coloneqq\bigcup_{i=1}^n E_i\) for every \(n\in\N\) and notice that \(\1_{E\setminus E'_n}\rightharpoonup 0\) weakly\(^*\)
in \(L^\infty(\mm)\) as \(n\to\infty\). By repeatedly applying \(2{\rm a})\), we obtain for any \(n\in\N\) that
\[
\|\1_E\cdot v\|^p=\|\1_{E_1}\cdot v\|^p+\|\1_{E\setminus E_1}\cdot v\|^p=\ldots=\sum_{i=1}^n\|\1_{E_i}\cdot v\|^p+\|\1_{E\setminus E'_n}\cdot v\|^p,
\]
whence by letting \(n\to\infty\) and using \(2{\rm b})\) we conclude that the identity in \eqref{eq:when_norm_ind_ptwse_norm_aux1} is verified.

Given any element \(v\in\mathscr M\), we define the set-function \(\mu_v\colon\Sigma\to[0,+\infty]\) as
\[
\mu_v(E)\coloneqq\|\1_E\cdot v\|^p,\quad\text{ for every }E\in\Sigma.
\]
It follows from \eqref{eq:when_norm_ind_ptwse_norm_aux1} that \(\mu_v\) is \(\sigma\)-additive. Given any \(N\in\Sigma\) with \(\mm(N)=0\), it holds
\(\1_N=0\) as elements of \(L^\infty(\mm)\), thus \(\mu_v(N)=\|0\cdot v\|^p=0\). Moreover, \(\mu_v(\X)=\|v\|^p<+\infty\). All in all, we have proven
that \(\mu_v\) is a finite measure on \(\Sigma\) satisfying \(\mu_v\ll\mm\). Hence, we can define
\[
|v|\coloneqq\bigg(\frac{\d\mu_v}{\d\mm}\bigg)^{1/p}\in L^p(\mm),\quad\text{ for every }v\in\mathscr M.
\]
Observe that \(\int|v|^p\,\d\mm=\mu_v(\X)=\|v\|^p\), thus in order to conclude it only remains to show that \(|\cdot|\colon\mathscr M\to L^p(\mm)\)
is a pointwise norm operator. Trivially, \(|v|=0\) holds \(\mm\)-a.e.\ if and only if \(v=0\). The \(\mm\)-a.e.\ inequality \(|v+w|\leq|v|+|w|\)
stems from Lemma \ref{lem:ineq_meas}: for \(E\in\Sigma\) we have
\[
\mu_{v+w}(E)^{1/p}=\big\|\1_E\cdot(v+w)\big\|=\big\|\1_E\cdot v+\1_E\cdot w\big\|\leq\|\1_E\cdot v\|+\|\1_E\cdot w\|=\mu_v(E)^{1/p}+\mu_w(E)^{1/p},
\]
thus Lemma \ref{lem:ineq_meas} ensures that \(|v+w|\leq|v|+|w|\) holds \(\mm\)-a.e.\ on \(\X\). Finally, we claim that
\begin{equation}\label{eq:when_norm_ind_ptwse_norm_aux2}
|f\cdot v|=|f||v|,\quad\text{ holds }\mm\text{-a.e.\ on }\X
\end{equation}
for every \(f\in L^\infty(\mm)\) and \(v\in\mathscr M\). Let us first prove it in the case where \(f\) is a simple function, namely,
\(f=\sum_{i=1}^n\lambda_i\,\1_{E_i}\) for some \(\lambda_1,\ldots,\lambda_n\in\R\) and pairwise disjoint sets \(E_1,\ldots,E_n\in\Sigma\).
To this aim, notice that for any set \(F\in\Sigma\) the following identities are satisfied:
\[\begin{split}
\int_F|f\cdot v|^p\,\d\mm&=\sum_{i=1}^n\int_{F\cap E_i}|f\cdot v|^p\,\d\mm=\sum_{i=1}^n\mu_{f\cdot v}(F\cap E_i)=\sum_{i=1}^n\big\|\1_{F\cap E_i}\cdot(f\cdot v)\big\|^p\\
&=\sum_{i=1}^n\big\|\lambda_i(\1_{F\cap E_i}\cdot v)\big\|^p=\sum_{i=1}^n|\lambda_i|^p\|\1_{F\cap E_i}\cdot v\|^p=\sum_{i=1}^n|\lambda_i|^p\int_F\1_{E_i}|v|^p\,\d\mm\\
&=\int_F|f|^p|v|^p\,\d\mm.
\end{split}\]
By arbitrariness of \(F\), we deduce that \eqref{eq:when_norm_ind_ptwse_norm_aux2} holds whenever \(f\) is a simple function. The general case follows by approximation:
given any \(f\in L^\infty(\mm)\), we can find a sequence \((f_n)_{n\in\N}\) of simple functions such that \(f_n\to f\) strongly in \(L^\infty(\mm)\) as \(n\to\infty\).
In particular, \(f_n\rightharpoonup f\) weakly\(^*\) in \(L^\infty(\mm)\), thus \(2{\rm b})\) yields
\[
\int\big||f_n\cdot v|-|f\cdot v|\big|^p\,\d\mm\leq\int\big|(f_n-f)\cdot v\big|^p\,\d\mm=\big\|(f_n-f)\cdot v\big\|^p\longrightarrow 0,\quad\text{ as }n\to\infty.
\]
Moreover, since \(|f_n|\to|f|\) in \(L^\infty(\mm)\), we have \(|f_n||v|\to|f||v|\) in \(L^p(\mm)\). Since we already know
that \(|f_n\cdot v|=|f_n||v|\) for all \(n\in\N\), we conclude that \(|f\cdot v|=\lim_n|f_n\cdot v|=\lim_n|f_n||v|=|f||v|\) strongly in \(L^p(\mm)\),
proving \eqref{eq:when_norm_ind_ptwse_norm_aux2}. Therefore, \(|\cdot|\) is a pointwise norm, whence \(1)\) follows.
\end{proof}
\def\cprime{$'$} \def\cprime{$'$}

\end{document}